\documentclass[11pt,reqno]{amsart}
\usepackage[margin=1.05in]{geometry}
\usepackage[T1]{fontenc}
\usepackage{lmodern,microtype}
\usepackage{etoolbox}
\usepackage{amsmath,amssymb,amsthm,mathtools,mathrsfs}
\usepackage{booktabs,array,longtable,enumitem}
\usepackage{needspace}
\usepackage{xcolor}
\usepackage[colorlinks=true,linkcolor=blue!45!black,citecolor=blue!45!black,urlcolor=blue!45!black]{hyperref}
\hypersetup{pdftitle={Point Corrections and Wall-Crossing Geometry on the Quintic del Pezzo Threefold},pdfsubject={Point corrections and wall crossing on the quintic del Pezzo threefold},pdfauthor={Zheyuan Fu}}
\newtheorem{theorem}{Theorem}[section]
\newtheorem{mainresult}{Theorem}

\newtheorem{proposition}[theorem]{Proposition}
\newtheorem{lemma}[theorem]{Lemma}
\newtheorem{corollary}[theorem]{Corollary}
\theoremstyle{definition}
\newtheorem{definition}[theorem]{Definition}

\theoremstyle{remark}
\newtheorem{remark}[theorem]{Remark}
\newcommand{\OO}{\mathcal O}
\newcommand{\PP}{\mathbf P}
\newcommand{\CC}{\mathbb C}
\newcommand{\ZZ}{\mathbb Z}
\newcommand{\cA}{\mathcal A}
\newcommand{\cE}{\mathcal E}
\newcommand{\cT}{\mathcal T}
\newcommand{\cI}{\mathcal I}
\newcommand{\cL}{\mathcal L}
\newcommand{\cS}{\mathcal S}

\newcommand{\cF}{\mathcal F}
\newcommand{\cU}{\mathcal U}
\newcommand{\cQ}{\mathcal Q}
\newcommand{\pt}{[\mathrm{pt}]}
\newcommand{\RHom}{R\mathcal Hom}
\newcommand{\Hom}{\operatorname{Hom}}
\newcommand{\Ext}{\operatorname{Ext}}
\newcommand{\ch}{\operatorname{ch}}
\newcommand{\rk}{\operatorname{rk}}
\newcommand{\Coh}{\operatorname{Coh}}

\newcommand{\Gr}{\operatorname{Gr}}
\newcommand{\Hilb}{\operatorname{Hilb}}
\newcommand{\Bl}{\operatorname{Bl}}

\newcommand{\Spf}{\operatorname{Spf}}
\newcommand{\Proj}{\operatorname{Proj}}
\newcommand{\Sym}{\operatorname{Sym}}
\newcommand{\Sing}{\operatorname{Sing}}
\newcommand{\Tot}{\operatorname{Tot}}
\newcommand{\Flag}{\operatorname{Flag}}
\newcommand{\Mat}{\operatorname{Mat}}
\newcommand{\cExt}{\mathcal E xt}

\newcommand{\im}{\operatorname{im}}
\newcommand{\id}{\operatorname{id}}

\newcommand{\bd}{b_d}
\newcommand{\ad}{a_d}
\newcommand{\Ntilde}{\widetilde N_5}
\newcommand{\Rtilde}{\widetilde R_5}

\setlist[enumerate]{label=\textup{(\roman*)},leftmargin=2em}
\setlist[itemize]{leftmargin=1.6em}
\numberwithin{equation}{section}
\BeforeBeginEnvironment{theorem}{\Needspace{5\baselineskip}}
\BeforeBeginEnvironment{mainresult}{\Needspace{5\baselineskip}}
\BeforeBeginEnvironment{proposition}{\Needspace{4\baselineskip}}
\BeforeBeginEnvironment{lemma}{\Needspace{4\baselineskip}}
\BeforeBeginEnvironment{corollary}{\Needspace{4\baselineskip}}
\AtBeginEnvironment{thebibliography}{}
\title[Point corrections and wall-crossing geometry on $V_5$]{Point Corrections and Wall-Crossing Geometry on the Quintic del Pezzo Threefold}
\makeatletter
\patchcmd{\@setauthors}{\MakeUppercase{\authors}}{\authors}{}{}
\makeatother
\author{Zheyuan Fu}
\address{School of Mathematics, Shandong University, Jinan 250100, P. R. China}
\email{202511781@mail.sdu.edu.cn}
\date{}

\begin{document}
\begin{abstract}
We study wall crossing for ideal classes of rational curves of degrees two through five on the quintic del Pezzo threefold $X=V_5$. Wall exclusion and a uniform bound on the correction length rule out point corrections in degrees two through four on the outer segments of the negative zero-tilt-slope paths. In degree five a correction of length one occurs and gives rise to a second component of the good moduli space.

For a rational outer stability parameter, this space is reduced and has exactly two irreducible components. The closure of the smooth rational-quintic locus is a normal tenfold. The other component is the smooth thirteenfold $X\times\Gr(2,7)$. Its points represent the polystable objects $I_{E_L}[1]\oplus\OO_p$, where the pencil $L$ cuts out a quintic linear section $E_L$ of arithmetic genus one. The scheme-theoretic intersection is the normal determinantal ninefold parametrizing pairs $(p,L)$ for which $E_L$ is singular at $p$. Along their common reduced singular locus, the principal component and the intersection have transverse models given by the cones of $2\times4$ and $2\times3$ matrices of rank at most one, respectively. We construct a small resolution of the principal component as the projectivization of a rank-five extension bundle over the blow-up of the Altavilla--Petkovi\'c--Rota sixfold along its point locus.
\end{abstract}
\maketitle

\section{Introduction}\label{sec:intro}
A smooth rational quintic $D$ on the quintic del Pezzo threefold $X=V_5$ has Hilbert polynomial $5t+1$. A codimension-two linear section $E$ of $X$ has degree five and arithmetic genus one, so its Hilbert polynomial is $5t$. The difference is the length of a single point. For any $p\in X$, the shifted ideal $I_D[1]$ and the direct sum $I_E[1]\oplus\OO_p$ therefore have the same numerical class. We show that these direct sums are polystable at the outer stability parameters considered below and fill a second irreducible component of the moduli space. Such a point lies in the closure of the smooth rational-quintic locus precisely when $E$ is singular at $p$.

The spaces of conics, twisted cubics and rational quartics on $X$ admit explicit descriptions \cite{Sanna,CHL,CKK}. To see why the additional component appears in degree five, we compare the corresponding negative zero-tilt-slope paths in degrees two through five. On the outer segment, the length of the zero-dimensional quotient in the second tilt measures the point correction. It vanishes in degrees two through four. In degree five it can equal one, and the ideals of the linear sections above exhaust the corrected factors.

\smallskip\noindent\textbf{The degree-five moduli space.}
Work over $\CC$ and let $H$ generate $\operatorname{Pic}(X)$, with $-K_X=2H$ and $H^3=5$. Write $\ch(F)=(r,cH,qH^2,eH^3)=(r,c,q,e)$ and $\pt=(0,0,0,1/5)$. The ideal of a rational quintic has class $v_5=(1,0,-1,4/5)$. Fix a rational outer stability parameter
\[
\sigma=\sigma_{\sqrt{b^2-2},-b,s},\qquad b>3/2,\quad s>0.
\]
Write $V=H^0(X,\OO_X(H))$, of dimension seven, and let $E_L$ be the base locus of a pencil $L\in\Gr(2,V)$. We use the subspace convention for projective bundles and Grassmannians.

The incidence condition is expressed by the evaluation and first-jet sequences
\begin{align}
0&\longrightarrow W\longrightarrow V\otimes\OO_X
 \longrightarrow\OO_X(H)\longrightarrow0,\label{eq:W}\\
0&\longrightarrow K\longrightarrow V\otimes\OO_X
 \longrightarrow J^1\OO_X(H)\longrightarrow0.\label{eq:K}
\end{align}
The inclusion $L\subset W_p$ means that $p\in E_L$. At such a point the two defining differentials are linearly dependent exactly when $L\cap K_p\ne0$. We give
\[
\Sigma=\{(p,L)\in X\times\Gr(2,V): L\subset W_p,\ \dim(L\cap K_p)\ge1\}
\]
its determinantal scheme structure and write $Z=\Gr_X(2,K)\subset\Sigma$ for the locus where both differentials vanish. Unless otherwise specified, singular loci carry their reduced structure.
\begin{mainresult}\label{thm:mainambient}
The semistable stack of class $-v_5$ has a proper good moduli space $M^{ss}=M_\sigma(-v_5)$, which is reduced and has exactly two irreducible components:
\begin{enumerate}
\item The principal component $R$, the reduced closure of the smooth rational-quintic locus, is a normal tenfold, Cohen--Macaulay with rational singularities. Its reduced singular locus is $Z$, of codimension five. At $z\in Z$, its completed local ring is
\[
\widehat{\OO}_{R,z}\simeq\widehat{\OO}_{Z,z}
 [[x_{11},x_{12},x_{13},x_{14},x_{21},x_{22},x_{23},x_{24}]]/I_2(\mathbf x),
\]
where $\mathbf x=(x_{ij})$ is a $2\times4$ matrix.
\item The corrected component is the smooth thirteenfold
\[
C=X\times\Gr(2,V)\hookrightarrow M^{ss},\qquad
(p,L)\longmapsto[I_{E_L}[1]\oplus\OO_p].
\]
\item The scheme-theoretic intersection is $R\times_{M^{ss}}C=\Sigma$. It is an irreducible normal Cohen--Macaulay ninefold with reduced singular locus $Z$. Transversely to $Z$, its singularities are the rank-at-most-one locus of a $2\times3$ matrix.
\end{enumerate}
\end{mainresult}
\smallskip\noindent\textbf{Construction of the principal branch.}
On the open locus used to construct $R$, the quintic $D$ lies on a hyperplane section $i:S\hookrightarrow X$. Its ideal fits into the sequence
\[
0\longrightarrow\OO_X(-H)\longrightarrow I_D
  \longrightarrow i_*I_{D/S}\longrightarrow0.
\]
Thus the residual sheaf $i_*I_{D/S}$ and an extension line recover $I_D$. The residual class is $\tau_5=v_5-\ch(\OO_X(-H))$. After twisting by $H$, it has class $w=e^H\tau_5$, to which the chamber comparison of Altavilla--Petkovi\'c--Rota applies \cite[Propositions~4.11--4.12]{APR}. Their contraction $\rho:\mathcal R_{\mathrm{root}}\to N$ maps the root locus to the smooth $3$-Kronecker sixfold $N=N^{\mathrm{Ku}}_5$ with point locus $j:X\hookrightarrow N$. Away from $j(X)$ the extension lines form a $\PP^4$-bundle. The following theorem describes how this family extends over the boundary.
\begin{mainresult}\label{thm:mainconstruction}
\begin{enumerate}
\item The point locus has normal bundle
\begin{equation}\label{eq:mainnormal}
N_{X/N}\simeq K\otimes\OO_X(H).
\end{equation}
The normalization of the reduced graph closure of $\rho^{-1}$ is $\Ntilde=\Bl_{j(X)}N$, with exceptional divisor $J\simeq\PP_X(K)$.
\item The residual universal family on $\Ntilde$ determines a rank-five extension bundle $\mathcal E_5$ and a canonical smooth tenfold $\Rtilde=\PP_{\Ntilde}(\mathcal E_5)$. Universal extensions induce a small resolution
\[
\Phi:\Rtilde\longrightarrow R.
\]
Over $J$, the projective bundle is the incidence resolution
\[
\widetilde\Sigma=\{(p,\ell,L):\ell\subset K_p,\ 
                         \ell\subset L\subset W_p\}\longrightarrow\Sigma.
\]
The restriction of $\Phi$ is $(p,\ell,L)\mapsto(p,L)$, followed by $\Sigma\hookrightarrow R$. Its exceptional locus is $\Flag_X(1,2;K)$, and its only positive-dimensional fibers are $\PP^1$ over $Z$. Away from this locus, $\Phi$ is an isomorphism.
\end{enumerate}
\end{mainresult}

\smallskip\noindent\textbf{The first point-correction threshold.}
For a curve with Hilbert polynomial $dm+1$, write
\begin{equation}\label{eq:vd}
v_d=\left(1,0,-\frac d5,\frac{d-1}{5}\right).
\end{equation}
For $s>0$ consider
\begin{equation}\label{eq:path}
\gamma_{d,s}(b)=\sigma_{\sqrt{b^2-2d/5},-b,s},\qquad b>\sqrt{2d/5},
\end{equation}
and set
\begin{equation}\label{eq:terminal}
\bd=\frac12+\frac d5,\qquad\ad=|\bd-1|,\qquad p_d=(\ad,-\bd).
\end{equation}
The outer and inner segments are $b>\bd$ and $\sqrt{2d/5}<b<\bd$, respectively. Strictly semistable objects are retained throughout.

\begin{mainresult}\label{thm:mainnumerical}
Let $2\le d\le5$ and $s>0$.
\begin{enumerate}
\item No actual tilt wall with finite-slope factors strictly encloses
\[
W_d:\quad\alpha^2+(\beta+\bd)^2=\ad^2
\]
in $\beta<0$. This assertion is independent of the third Chern character.
\item Every object $E$ of class $-v_d$ in the second-tilt heart on the outer segment fits into
\[
0\longrightarrow F[1]\longrightarrow E\longrightarrow Q\longrightarrow0,
\quad\ch(F)=v_d+m\pt,\quad\operatorname{length}(Q)=m,
\]
where $F$ is tilt-semistable of slope zero and
\[
2d^2-10d+15-15m\ge0.
\]
Consequently $m=0$ for $d\le4$ and $m\le1$ for $d=5$; if $m>0$, then $E$ is strictly Bridgeland-semistable.
\item On the inner segment there is no tilt-semistable object of class $v_d+m\pt$ for any $m\ge0$, and no object of class $-v_d$ in the second-tilt heart.
\item In degree five, every outer tilt-semistable object of class $u_5=v_5+\pt$ is tilt-stable and is $I_{E_L}$ for a unique $L\in\Gr(2,V)$. Conversely, every $I_{E_L}[1]$ is stable at every outer parameter, and $I_{E_L}[1]\oplus\OO_p$ is polystable of class $-v_5$. Thus degree five is the first actual positive point-correction threshold.
\end{enumerate}
\end{mainresult}
\smallskip\noindent\textbf{Proof of the main results.}
Sections~\ref{sec:hearts}--\ref{sec:budget} exclude larger finite-slope tilt walls and transport the outer factors to the terminal parameter, where a numerical inequality bounds the length of the point correction. In degree five a length-one correction forces equality in this bound. Section~\ref{sec:actual} classifies the factors in the equality case and recovers the Koszul presentation of $I_{E_L}$ from a pencil $L$.

For the principal component, Sections~\ref{sec:normal}--\ref{sec:extensions} identify the normalized inverse graph of $\rho$ with the blow-up of $N$ and extend the rank-five bundle of extension classes across its exceptional divisor. Over that divisor an extension line specifies a triple $(p,\ell,L)$. Its polystable representative $I_{E_L}[1]\oplus\OO_p$ depends only on $(p,L)$, so the boundary map forgets $\ell$. The choices of $\ell$ form a $\PP^1$ exactly over $Z$.

Section~\ref{sec:det} constructs $\Sigma$ from first jets. To identify it with the scheme-theoretic intersection in the moduli space, Section~\ref{sec:ambient} computes the completed local ring at a polystable point. The two branch ideals intersect in the defining ideal of the union, and their sum cuts out $\Sigma$. Together with the reconstruction on the stable locus, this calculation proves that the moduli space is reduced and that the small resolution maps directly to its normal principal component (Theorem~\ref{thm:ambientgluing}).
\section{Hearts, central charges, and phase-one objects}\label{sec:hearts}
Riemann--Roch is
\begin{equation}\label{eq:RR}
\chi(F)=r+\frac83c+5q+5e.
\end{equation}
Write
\[
c^\beta=c-\beta r,\qquad
q^\beta=q-\beta c+\frac{\beta^2}{2}r,\qquad
e^\beta=e-\beta q+\frac{\beta^2}{2}c-\frac{\beta^3}{6}r.
\]
We use the conventions for stability conditions, slicings, and locally finite slicings of \cite{Bridgeland}. The first-tilt heart is $\Coh^\beta(X)$ and the tilt slope is
\[
\nu_{\alpha,\beta}(F)=\frac{q^\beta(F)-\alpha^2r(F)/2}{c^\beta(F)}.
\]
Objects with denominator zero have slope $+\infty$. Tilt stability uses the inequality between the slope of a subobject and that of its quotient; this convention accommodates the zero-dimensional null category.

Let $\cF_{\alpha,\beta}$, respectively $\cT_{\alpha,\beta}$, consist of objects with tilt-Harder--Narasimhan slopes at most zero, respectively strictly positive, with the usual convention for the null category. The second-tilt heart is
\[
\cA_{\alpha,\beta}=\langle\cF_{\alpha,\beta}[1],\cT_{\alpha,\beta}\rangle.
\]
The central charge is
\begin{equation}\label{eq:central}
Z_{\alpha,\beta,s}(F)=-e^\beta(F)+\left(s+\frac16\right)\alpha^2c^\beta(F)
+i\left(q^\beta(F)-\frac{\alpha^2}{2}r(F)\right).
\end{equation}
For $X$ these are geometric stability conditions by the BMT construction and Li's inequality \cite{BMT,Li}. We use the standard discriminant inequality $\Delta(F)=c^2-2rq\ge0$ for tilt-semistable objects and the slope-zero consequence
\begin{equation}\label{eq:BMT}
e^\beta(F)\le\frac{\alpha^2}{6}c^\beta(F)
\qquad\text{if }\nu_{\alpha,\beta}(F)=0.
\end{equation}
We also use local finiteness of tilt walls, their nested structure, constancy in their complements, and large-$\alpha$ stability of slope-stable sheaves; see \cite[Proposition 3.2, Theorem 3.3, and Lemma 3.7]{Schmidt}. Line bundles, with the shift placing them in $\Coh^\beta(X)$, are tilt-stable; their discriminant is zero.

\smallskip\noindent\textbf{Moduli input.}\label{par:moduliexistence}
We use good moduli spaces in the sense of \cite[Definition 4.1 and Theorem 6.6]{AlperGMS}. The numerical lattice has finite rank. At rational $\beta,\alpha^2,s$, the charge is algebraic and the heart is noetherian with bounded standard amplitude \cite[Example 7.29]{AHLH}. Li's inequality places these conditions in the BMT component. Generic flatness and fixed-class boundedness follow from \cite[Corollary 4.21 and Proposition 4.12]{PT}; \cite[Theorem 1.5]{PT} gives the open finite-type semistable stack and semistable reduction. Its Harder--Narasimhan $\Theta$-stratification on the ind-completed heart yields a proper good moduli space by \cite[Theorem 7.27 and Example 7.29]{AHLH}.

For \cite[Theorems 4.12 and 4.16, Remark 4.17]{AHR}, the diagonal is affine: if $\mathcal E,\mathcal F$ are perfect families in the fixed heart over a noetherian base $T$, negative fiber Ext groups vanish. The perfect, base-change-compatible complex $R\pi_{T*}R\mathcal Hom(\mathcal E,\mathcal F)$ therefore has locally a finite free representative in nonnegative degrees. Relative Hom is the affine kernel of its first differential. Relative Isom is cut out in the product of the two relative Hom schemes by the inverse-composition equations, so it is affine. At a closed polystable object, the stabilizer is a product of general linear groups, linearly reductive over $\CC$. AHR thus identifies the completed good-moduli ring with the invariants of the equivariant miniversal ring constructed in Section~\ref{sec:ambient}.

\begin{lemma}\label{lem:phaseone}
Suppose $(\alpha,\beta)=(\sqrt{b^2-2d/5},-b)$, $b>\sqrt{2d/5}$, and $E\in\cA_{\alpha,\beta}$ has class $-v_d$. Then there is a short exact sequence in this heart
\begin{equation}\label{eq:decomp}
0\longrightarrow F[1]\longrightarrow E\longrightarrow Q\longrightarrow0,
\end{equation}
where $Q$ is zero-dimensional and $F$ is a nonzero tilt-semistable object of slope zero. If $m=\operatorname{length}(Q)$, then $\ch(F)=v_d+m\pt$. If $m>0$, $E$ is strictly semistable.
\end{lemma}
\begin{proof}
Take cohomology with respect to $\Coh^\beta(X)$. By the definition of the second tilt, $F=H^{-1}_{\Coh^\beta}(E)$ belongs to $\cF_{\alpha,\beta}$ and $Q=H^0_{\Coh^\beta}(E)$ belongs to $\cT_{\alpha,\beta}$. The imaginary part of $Z(E)$ is zero, and those of both terms in \eqref{eq:decomp} are nonnegative. They must vanish separately. A positive-slope term has strictly positive imaginary part unless it lies in the null category; that category consists of zero-dimensional sheaves. All nonzero HN factors of $F$ have slope zero. This proves the structural assertions and the Chern-character identity.

Every nonzero heart object with zero imaginary part has negative real central charge. Consequently $F[1]$, $E$, and every nonzero $Q$ have phase one. A heart object of phase one is semistable. If $Q\ne0$, the first term of \eqref{eq:decomp} is a nonzero proper subobject of the same phase, so $E$ is not stable.
\end{proof}

\begin{lemma}\label{lem:stablecriterion}
Let $F$ be tilt-stable of slope zero. If $\Hom(\OO_p,F[1])=0$ for all closed points $p\in X$, then $F[1]$ is Bridgeland-stable for every $s>0$.
\end{lemma}
\begin{proof}
The object $F[1]$ is in $\cA_{\alpha,\beta}$ and has phase one. Consider a short exact sequence $0\to B\to F[1]\to C\to0$ with both nonzero terms of phase one. The first-tilt cohomology sequence has the form
\[
0\to F_B\to F\to F_C\to Q_B\to0,
\qquad H^0_{\Coh^\beta}(C)=0,
\]
where $F_B,F_C$ have slope zero when nonzero and $Q_B$ is zero-dimensional. Put $J=\im(F\to F_C)$. If $J\ne0$, it cannot be zero-dimensional because it is a subobject of $F_C\in\cF_{\alpha,\beta}$, and its tilt slope is zero. Thus tilt stability excludes $F_B\ne0$ and $J\ne0$ simultaneously. If $J=0$, then $F_C\simeq Q_B$ belongs to both torsion-pair classes, hence vanishes, and $C=0$. If $F_B=0$, then $B=Q_B$ is zero-dimensional. Every nonzero finite-length sheaf contains a skyscraper subsheaf, contradicting the assumed Hom vanishing. No proper subobject of phase one exists.
\end{proof}

\section{The terminal wall and exclusion of larger walls}\label{sec:walls}
This section proves the wall-exclusion statement in Theorem~\ref{thm:mainnumerical}(i) and the transport used for its correction bound.

Equating the tilt slopes of $v_d$ and $\OO_X(-1)$ gives $W_d$ of Theorem~\ref{thm:mainnumerical}. The identity
\[
\bd^2-\frac{2d}{5}=(\bd-1)^2=\ad^2
\]
shows that $p_d$ lies on \eqref{eq:path}. The four vertices are
\[
(1/10,-9/10),\quad(1/10,-11/10),\quad(3/10,-13/10),\quad(1/2,-3/2).
\]
Every nonvertical numerical tilt wall for the first three entries of $v_d$ has the equation
\begin{equation}\label{eq:wallS}
\alpha^2+(\beta+S)^2=S^2-\frac{2d}{5}.
\end{equation}
Its intersection with the negative zero-slope branch has $b=S$.

\begin{proposition}\label{prop:nolarger}
For $2\le d\le5$, no actual tilt wall in $\beta<0$ with finite-slope factors strictly encloses $W_d$, for any prescribed third Chern character.
\end{proposition}
\begin{proof}
Suppose a larger wall exists, so $S>\bd$. Its intersection with $\beta=-\bd$ lies at $\alpha>0$. By the structure theorem for tilt walls, we can use a semistable decomposition there. After exchanging its two terms, write their first three Chern characters as
\[
A=(r,-k,q),\qquad B=(1-r,k,-d/5-q),\qquad r\ge1.
\]
Both denominators are positive, hence
\begin{equation}\label{eq:kinterval}
\bd(r-1)<k<\bd r.
\end{equation}
In particular $k>0$. The wall equation gives
\begin{equation}\label{eq:qS}
q=kS-\frac{dr}{5}>k\bd-\frac{dr}{5}.
\end{equation}
If $r\ge2$, put $x=\bd r-k$ and $y=k-\bd(r-1)$. Then $x,y>0$, $x+y=\bd$, and
\begin{align*}
\Delta(A)&=x^2-\ad^2r^2-2rk(S-\bd),\\
\Delta(B)&=y^2-\ad^2(r-1)^2-2(r-1)k(S-\bd).
\end{align*}
The two Bogomolov inequalities imply $x>\ad r$, $y>\ad(r-1)$, and consequently
\begin{equation}\label{eq:rankbound}
(2r-1)\ad<\bd.
\end{equation}
This leaves exactly the rows of Table~\ref{tab:larger}. In each row the necessary inequalities for $q$ are
\[
k\bd-\frac{dr}{5}<q\le\frac{k^2}{2r},
\qquad q\in\frac{k^2}{2}-\frac15\ZZ.
\]
The lattice condition is the integrality of $c_2(A)$, since $H^2$ is five times the integral degree-four generator. The intervals contain no such lattice point: for $d=2$ the lower endpoint belongs to the lattice and the next lattice point is already larger than the upper endpoint; in the other rows the strict lower and weak upper bounds coincide. For $d=5$, \eqref{eq:rankbound} leaves no rank $r\ge2$.

Finally, if $r=1$, \eqref{eq:kinterval} becomes $0<k<\bd$. There is no integer $k$ for $d=2$. For $d=3,4,5$ the only possibility is $k=1$, and $\Delta(A)\ge0$ gives $q\le1/2$. Thus $S=q+d/5\le\bd$, again a contradiction. Rank-zero terms were included as the complementary term in this case. A term in the zero-dimensional null category does not define a finite-slope semicircular wall and does not affect this argument.
\end{proof}

\begin{table}[ht]
\centering
\caption{All higher-rank possibilities for a strictly enclosing wall.}\label{tab:larger}
\begin{tabular}{ccccc}
\toprule
$d$&$r$&$k$&strict lower bound for $q$&weak upper bound for $q$\\
\midrule
2&2&1&$1/10$&$1/4$\\
2&3&2&$3/5$&$2/3$\\
2&4&3&$11/10$&$9/8$\\
3&2&2&$1$&$1$\\
3&3&3&$3/2$&$3/2$\\
3&4&4&$2$&$2$\\
3&5&5&$5/2$&$5/2$\\
4&2&2&$1$&$1$\\
\bottomrule
\end{tabular}
\end{table}

\begin{corollary}\label{cor:transport}
A tilt-semistable object on the outer segment whose first three entries are $(1,0,-d/5)$ remains tilt-semistable when transported to $p_d$. There are no actual tilt walls on the open outer segment.
\end{corollary}
\begin{proof}
Any wall meeting that segment has $S=b>\bd$ in \eqref{eq:wallS}, and hence is excluded by Proposition~\ref{prop:nolarger}. There is no vertical wall with $\beta<0$. Apply local finiteness, constancy away from walls, and closedness of semistability at the endpoint.
\end{proof}

\section{The correction budget and inner emptiness}\label{sec:budget}
This section proves Theorem~\ref{thm:mainnumerical}(ii)--(iii): the correction budget distinguishes degree five from the lower degrees, and the inner segment is empty.

At $p_d$ let $T=(r,c,q,e)$ have finite tilt slope zero. Then
\begin{equation}\label{eq:terminalrelations}
x=c+\bd r>0,\qquad q=-\bd c-\frac d5r,
\qquad \Delta(T)=x^2-\ad^2r^2.
\end{equation}
Define the additive linear functional
\begin{equation}\label{eq:remainder}
D_d(T)=\frac{\ad^2}{6}(c+\bd r)-e^{-\bd}(T).
\end{equation}
It is nonnegative for tilt-semistable $T$ of slope zero by \eqref{eq:BMT}.

\begin{proposition}\label{prop:budget}
For slope-zero classes at $p_d$,
\begin{equation}\label{eq:RRbudget}
D_d(T)=\frac{(2d^2-10d+15)(c+r)-15\chi(T)}{75}.
\end{equation}
Consequently
\begin{equation}\label{eq:budget}
D_d(v_d+m\pt)=\frac{2d^2-10d+15-15m}{75}.
\end{equation}
If a slope-zero tilt-semistable object admits a finite filtration with slope-zero tilt-semistable factors $T_i$, then $D_d(F)=\sum_iD_d(T_i)$. In particular, $D_d(F)=0$ implies $D_d(T_i)=0$ for all $i$.
\end{proposition}
\begin{proof}
Insert \eqref{eq:terminalrelations} in \eqref{eq:remainder} and use \eqref{eq:RR}. For $v_d+m\pt$ we have $c+r=1$ and $\chi=m$, proving \eqref{eq:budget}. Additivity follows from the Chern character; nonnegativity of each summand follows from \eqref{eq:BMT}.
\end{proof}
The finite filtrations used below are constructed in the genuine phase-one categories of Section~\ref{sec:kronecker}.

\begin{corollary}\label{cor:budgetouter}
The correction length in Lemma~\ref{lem:phaseone} is zero for $d=2,3,4$ and at most one for $d=5$. For $d=5$, a terminal corrected object of class $v_5+\pt$ has zero total remainder.
\end{corollary}
\begin{proof}
Transport $F$ to $p_d$ using Corollary~\ref{cor:transport}, apply \eqref{eq:BMT}, and use \eqref{eq:budget}.
\end{proof}
\begin{proposition}\label{prop:inner}
For $2\le d\le5$ and $\sqrt{2d/5}<b<\bd$, no tilt-semistable object of class $v_d+m\pt$ exists for any $m\ge0$.
\end{proposition}
\begin{proof}
Suppose $F$ is such an object. Its tilt slope is zero. For $d=2$ take $A=\OO_X(-1)[1]\in\Coh^{-b}(X)$. Then
\[
\nu(A)=\frac{b-9/10}{1-b}<0,\qquad
\nu(\OO_X(1))=\frac{b+9/10}{b+1}>0.
\]
Tilt-stability of the line-bundle objects gives $\Hom(F,A)=0$ and $\Hom(\OO_X(1),F)=0$. Serre duality gives
\[
\Ext^2(F,A)\simeq\Hom(A,F(-2)[1])^\vee
\simeq\Hom(\OO_X(1),F)^\vee=0.
\]

For $d=3,4,5$ take $A=\OO_X(-3)[1]\in\Coh^{-b}(X)$. Here $1<b<\bd\le3/2$, and
\[
\nu(A)=\frac{3b-9/2-d/5}{3-b}<0,\qquad
\nu(\OO_X(-1))=\frac{\bd-b}{b-1}>0.
\]
The same argument gives $\Hom(F,A)=0$ and
\[
\Ext^2(F,A)\simeq\Hom(\OO_X(-1),F)^\vee=0.
\]
Because $F$ has coherent-sheaf amplitude $[-1,0]$ and $A$ is a line bundle shifted by one, $\Ext^i(F,A)=0$ for $i\ge4$. Negative Ext groups vanish in the common heart. Therefore $\chi(F,A)=-\dim\Ext^1(F,A)-\dim\Ext^3(F,A)\le0$. On the other hand, Riemann--Roch gives
\[
\chi(F,\OO_X(-1)[1])=d-1+m,\qquad
\chi(F,\OO_X(-3)[1])=6-d+m.
\]
The relevant values are $1+m,3+m,2+m,1+m$, all positive, a contradiction.
\end{proof}
Lemma~\ref{lem:phaseone} now implies the second-tilt-heart assertion in Theorem~\ref{thm:mainnumerical}(iii).

\section{The corrected degree-five factors}\label{sec:actual}\label{sec:kronecker}
We construct and exhaust the corrected factors of Theorem~\ref{thm:mainnumerical}(iv).

Put
\[
u_5=v_5+\pt=(1,0,-1,1),\qquad A=\OO_X(-1),\qquad B=\OO_X(-2)[1].
\]
Then $u_5=2\ch(A)+\ch(B)$, and both $A$ and $B$ have slope zero and zero remainder at $p_5$. Moreover
\begin{equation}\label{eq:sevenarrows}
\Ext^1(B,A)=H^0(X,\OO_X(H))=V,\qquad \Ext^1(A,B)=0.
\end{equation}
Proposition~\ref{prop:d5zerofactors} proves that these two objects exhaust the stable zero-remainder factors of $u_5$.

\begin{proposition}\label{prop:CI}
For every $L\in\Gr(2,V)$ the base locus $E_L$ is a local complete intersection curve with fundamental class $H^2$ and arithmetic genus one. Its ideal has resolution
\begin{equation}\label{eq:koszul}
0\longrightarrow\OO_X(-2)\longrightarrow\OO_X(-1)^{\oplus2}
\longrightarrow I_{E_L}\longrightarrow0.
\end{equation}
For every outer parameter on $\gamma_{5,s}$, $I_{E_L}[1]$ is Bridgeland-stable.
\end{proposition}
\begin{proof}
Two independent hyperplanes cannot have a common divisorial component on $X$: any effective divisor has positive integral class in $\ZZ H$, so a hyperplane divisor is irreducible and reduced, and a common component would make the divisors equal. Their equations thus form a regular sequence. The Koszul resolution proves \eqref{eq:koszul} and $\ch(I_{E_L})=u_5$. Riemann--Roch gives $\chi(\OO_{E_L}(m))=5m$.

The rank-one torsion-free sheaf $I_{E_L}$ is slope-stable, so it is tilt-stable for sufficiently large $\alpha$ with fixed $\beta<0$. At $\beta=-b$, $b>3/2$, its vertical path down to $\alpha=\sqrt{b^2-2}$ lies outside $W_5$. Nested walls and Proposition~\ref{prop:nolarger} exclude any actual tilt wall on this route. Hence $I_{E_L}$ is tilt-stable there, with slope zero.

At a point of $E_L$, the ideal has depth two over the regular local ring of the smooth threefold; off $E_L$ it is locally free. Thus $\Ext^1_X(\OO_p,I_{E_L})=0$ for every $p$. Apply Lemma~\ref{lem:stablecriterion}.
\end{proof}

\begin{corollary}\label{cor:actual}
For each $b>3/2$ and $s>0$, the universal linear-section family and the universal point family give polystable objects
\begin{equation}\label{eq:polystable}
I_{E_L}[1]\oplus\OO_p
\end{equation}
of class $-v_5$ over $X\times\Gr(2,V)$. Distinct pairs have distinct S-equivalence classes. There are no positive point corrections on the corresponding outer paths for $d=2,3,4$.
\end{corollary}
\begin{proof}
The two summands in \eqref{eq:polystable} are stable of phase one and have different numerical classes. Their direct sum has class $-u_5+\pt=-v_5$. Its unordered stable factors recover $p$ and $I_{E_L}$. The latter recovers $L=H^0(I_{E_L}(H))$ by \eqref{eq:koszul}. The Koszul construction has constant fiber Hilbert polynomial and gives the usual flat universal family of linear sections. The last assertion is Corollary~\ref{cor:budgetouter}.
\end{proof}

\subsection{The finite-length category at the degree-five terminal vertex}
Fix $s>0$ and let $\mathcal P_5(1)$ be the phase-one category of the genuine stability condition $\sigma_{a_5,-b_5,s}$. It is abelian and of finite length by local finiteness of the slicing \cite[Section 5]{Bridgeland}. We work in a Serre subcategory of this category, rather than imposing a Jordan--H\"older theory on the entire weak tilt-slope category.

The proof of Lemma~\ref{lem:phaseone}, which uses only $\Im Z=0$, gives for every $G\in\mathcal P_5(1)$ an exact sequence
\[
0\longrightarrow T[1]\longrightarrow G\longrightarrow Q\longrightarrow0,
\]
where $T=0$ or $T$ is tilt-semistable of slope zero, and $Q$ is zero-dimensional. For the additive functional $\delta(G):=D_5(-\ch(G))$, the BMT inequality gives
\[
\delta(G)=D_5(T)+\frac{\operatorname{length}(Q)}5\ge0.
\]
Consequently $\mathcal P_5(1)_{\delta=0}$ is a Serre subcategory. Its objects are exactly the shifts $T[1]$ with $\nu_{p_5}(T)=0$ and $D_5(T)=0$, and every subquotient has the same form. Shifting a Jordan--H\"older filtration back by one gives an exact filtration in $\Coh^{-b_5}(X)$ without zero-dimensional error terms. Its simple factors correspond to tilt-stable objects of zero slope and zero remainder: any equal-finite-slope destabilizing sequence would, by additivity and nonnegativity of the remainder, give a proper subobject in the same Serre subcategory. Conversely, if $T$ is tilt-stable with zero slope and remainder, then $T[1]$ is simple in this Serre subcategory. Indeed, any phase-one subobject and quotient have zero $\delta$, hence are shifts $T_1[1],T_2[1]$ without point terms; their sequence shifts back to an equal-finite-slope sequence for $T$, contradicting tilt stability unless one term is zero. Since the subcategory is Serre, these objects are also simple in $\mathcal P_5(1)$. This gives the finite stable filtration and its Jordan--H\"older uniqueness.

Every nonzero unshifted factor satisfies
\[
x(T)=c(T)+\frac32r(T)>0,\qquad 2x(T)\in\ZZ_{>0}.
\]
For $u_5$ one has $x(u_5)=3/2$, so such a filtration has at most three factors. This numerical bound supplements the categorical construction.

\subsection{The zero-remainder category in degree five}
The zero-remainder equality imposes the following rigid classification.

\begin{proposition}\label{prop:d5zerofactors}
Let $T$ be a stable finite-slope object at $p_5$ with $\nu_{p_5}(T)=0$ and
$D_5(T)=0$, and suppose that $T$ occurs in a finite stable filtration of $u_5$.
Then
\[
T\simeq A=\OO_X(-1)\qquad\text{or}\qquad
T\simeq B=\OO_X(-2)[1].
\]
Consequently every stable filtration of $u_5$ has factors $A,A,B$.
\end{proposition}
\begin{proof}
Write $T=(r,c,q,e)$ and set
\[
m=c+r,\qquad n=c+2r.
\]
At $d=5$, equation \eqref{eq:RRbudget} and $D_5(T)=0$ give $m=\chi(T)$. Combining
this with the slope-zero equation in \eqref{eq:terminalrelations} yields
\begin{equation}\label{eq:d5lattice}
(r,c,q,e)=\left(n-m,\,2m-n,\,\frac n2-2m,\,\frac{4m}{3}-\frac n6\right)
=n\ch(A)+m\ch(B).
\end{equation}
Moreover
\[
x(T)=\frac{m+n}{2}>0,\qquad \Delta(T)=mn\ge0.
\]
Since $m,n\in\ZZ$, these two inequalities imply $m,n\ge0$.

Assume $m,n>0$. Riemann--Roch gives
\[
\chi(A,T)=n>0.
\]
Both $A[1]$ and $T[1]$ are simple in $\mathcal P_5(1)_{\delta=0}$ by the preceding argument. For $m,n>0$ their classes differ, so
$\Hom(A,T)=\Hom(A[1],T[1])=0$.
This uses simplicity in the genuine phase-one category, rather than a Hom-vanishing assertion for arbitrary tilt-stable objects of equal weak slope. Also
\[
\Ext^2(A,T)\simeq\Hom(T,\OO_X(-3)[1])^\vee=0,
\]
because $\OO_X(-3)[1]$ has strictly negative tilt slope at $p_5$. The amplitude
bound therefore gives $\chi(A,T)\le0$, a contradiction. Thus $mn=0$.

If $m=0$, then \eqref{eq:d5lattice} gives $\ch(T)=n\ch(A)$; \cite[Proposition~3.1]{VassilievLow} gives $T\simeq A^{\oplus n}$, and stability forces $n=1$. If
$n=0$, \cite[Corollary~3.3]{VassilievLow} gives $T\simeq B^{\oplus m}$ and hence
$m=1$. Finally $u_5=2\ch(A)+\ch(B)$ fixes the multiplicities.
\end{proof}

\subsection{Classification of the corrected factors}
\begin{proposition}\label{prop:d5reconstruction}
Let $F$ be tilt-semistable of class $u_5=v_5+\pt$ at an outer point of $\gamma_{5,s}$. Then $F$ is tilt-stable, and there is a unique two-plane $L\subset V$ such that
\[
F\simeq I_{E_L}.
\]
Equivalently, the outer stable quotient for the corrected factor is $\Gr(2,V)=\Gr(2,7)$.
\end{proposition}
\begin{proof}
The first three Chern-character entries of $u_5$ are the same as those of $v_5$. Proposition~\ref{prop:nolarger} therefore excludes every actual tilt wall on the vertical ray from the chosen outer point to sufficiently large $\alpha$. At large $\alpha$, a tilt-semistable object of rank one and $c_1=0$ must be a torsion-free sheaf: if its ordinary $H^{-1}$ were nonzero, the subobject $H^{-1}[1]$ would have tilt slope tending to $+\infty$, whereas the slope of $F$ tends to $-\infty$; a torsion subsheaf of $H^0$ would likewise have finite or infinite tilt slope and destabilize. Thus $F$ itself is a rank-one torsion-free sheaf. Such a sheaf is slope-stable in rank one, hence tilt-stable for sufficiently large $\alpha$. Since there is no actual wall on the vertical ray, $F$ was already tilt-stable at the chosen outer point.

Transport $F$ to $p_5$. Proposition~\ref{prop:d5zerofactors}, finite length at the terminal slope, and $\Ext^1(A,B)=0$ reorder its terminal filtration into a triangle
\[
A^{\oplus2}\longrightarrow F\longrightarrow B\longrightarrow A^{\oplus2}[1],
\qquad A=\OO_X(-1),\quad B=\OO_X(-2)[1].
\]
After rotating, this is the cone of a morphism
\begin{equation}\label{eq:d5correctedmap}
\OO_X(-2)\xrightarrow{(s_1,s_2)}\OO_X(-1)^{\oplus2},
\qquad (s_1,s_2)\in V^{\oplus2}.
\end{equation}
If $s_1,s_2$ are linearly dependent, the cokernel of \eqref{eq:d5correctedmap} has a divisorial torsion summand (and if both vanish the cone also has nonzero negative cohomology), contradicting that $F$ is a torsion-free sheaf. Hence $s_1,s_2$ are independent. They form a regular sequence by Proposition~\ref{prop:CI}, and the Koszul complex identifies the cone with $I_{E_L}$ for $L=\langle s_1,s_2\rangle$. Finally $L=H^0(I_{E_L}(H))$, so it is unique. Modding out the choice of ordered basis by $\mathrm{GL}_2$ gives $\Gr(2,V)$.
\end{proof}

\begin{remark}\label{rem:lowdegreebackground}
The known degree-two, three, and four curve spaces are $\PP^4$, $\Gr(3,5)$, and $\Gr_{\PP^2}(2,\cF)$, where $\cF$ is the rank-five bundle $\pi_*\cI_{\mathcal L_{\mathrm{univ}}}(1)$ over the line space $\PP^2$. See \cite[Proposition 2.32, Corollary 2.33, and Proposition 2.46 with Remark 2.47]{Sanna}, \cite[Proposition 7.2]{CHL}, and \cite[Theorem 1.1 and Propositions 4.4--4.5]{CKK}. Their terminal reconstruction is not used here; Corollary~\ref{cor:budgetouter} supplies the absence of positive corrections.
\end{remark}

\section{The APR root sixfold and the inverse contraction graph}\label{sec:normal}
We use the APR chamber contraction to prove the normal-bundle and graph statements of Theorem~\ref{thm:mainconstruction}(i).

Set
\[
\tau_5=v_5-\ch(\OO_X(-1))=\left(0,1,-\frac32,\frac{29}{30}\right),
\qquad w=\tau_5(1)=\left(0,1,-\frac12,-\frac1{30}\right).
\]
The full strong exceptional collection originates in \cite{OrlovV5}; in the form recalled in \cite[Section 3]{APR}, it gives the presentation
\[
D^b(X)=\langle\cU,\cQ^\vee,\OO_X,\OO_X(1)\rangle,
\qquad w=2[\cQ^\vee]-3[\cU].
\]
We use the moduli space $N^{\mathrm{Ku}}_5$ and the embedding $j$ established in \cite[Section 5.1 and Lemmas 4.8--4.9]{APR}. Its points on $j(X)$ are right mutations $E_p$ fitting into triangles
\begin{equation}\label{eq:APRpointtriangle}
\OO_X(-1)[1]\longrightarrow E_p\longrightarrow I_p
\longrightarrow\OO_X(-1)[2].
\end{equation}
The scalar choice in the last arrow is harmless at one point but must be tracked in a family. The sixfold is smooth, rational, and Fano of Picard rank one and index three; see also \cite[Proposition 2.1]{Meng}.

\subsection{The APR chamber contraction}\label{subsec:exactwall}
Tensoring by $\OO_X(H)$ sends
\[
(p_5,\tau_5)=\left(\frac12,-\frac32;\tau_5\right)
\quad\text{to}\quad
\left(\frac12,-\frac12;w\right).
\]
The latter point is the unique weak tilt wall for $w$ in \cite[Proposition~4.7]{APR}. At that wall the Gieseker-side objects fit into triangles
\begin{equation}\label{eq:APRwall}
I_p\longrightarrow I_{p/S}\longrightarrow\OO_X(-1)[1],
\end{equation}
and the reverse non-split extension is $E_p$. We use this adjacent-chamber description in the geometric constructions below.

\begin{definition}\label{def:principalwall}
Let $\mathcal R_{\mathrm{root}}$ be the Gieseker root component of $M_H(w)$, and let $\rho$ be the restriction to it of the APR wall-crossing morphism. Throughout the geometric constructions below, $N=N^{\mathrm{Ku}}_5$ denotes the Kuznetsov-component moduli sixfold, identified with the adjacent chamber moduli $M_-(w)$.
\end{definition}

\begin{theorem}\label{thm:exactprincipal}
There is a canonical chamber isomorphism $M_-(w)\simeq N^{\mathrm{Ku}}_5$. The APR morphism restricts to a birational contraction
\[
\rho:\mathcal R_{\mathrm{root}}\longrightarrow N^{\mathrm{Ku}}_5.
\]
It is an isomorphism off the point boundary and sends $I_{p/S}$, for $S$ singular at $p$, to $E_p=j(p)$. The fiber of $\rho$ over $j(p)$ has underlying points parametrized by $\PP(K_p)$.
\end{theorem}
\begin{proof}
By \cite[Proposition~4.11]{APR}, relative mutation of the universal Gieseker family gives $s:M_H(w)\to M_-(w)$; it is an isomorphism off the point-on-hyperplane component and contracts its fiber with fixed $p$ to $E_p$. Proposition~4.12 and Section~5.1 of the same reference identify the target with the smooth $3$-Kronecker sixfold $N^{\mathrm{Ku}}_5$. Proposition~4.1 identifies the intersection with the root component as the sheaves $I_{p/S}$ for which $S$ is singular at $p$. Such sections are exactly the nonzero elements of $K_p$ up to scalar. Restricting $s$ gives the assertions. This uses the APR adjacent-chamber moduli functor.
\end{proof}

\subsection{Relative mutation and its line-bundle coefficient}
Let $p_1,p_2:X\times X\to X$ be the projections, with the second factor the parameter space, and put $\pi=p_2$. Let $\cI=I_\Delta$ be the universal point ideal and $A=\OO_X(-H)$. We use relative Ext sheaves
\[
\cExt^i_\pi(\mathcal B,\mathcal D)
=\mathcal H^i\bigl(R\pi_*R\mathcal Hom(\mathcal B,\mathcal D)\bigr).
\]

\begin{lemma}\label{lem:mutationline}
There is a natural isomorphism
\[
R\pi_*R\mathcal Hom(\cI,p_1^*A)\simeq\cL[-2],
\qquad \cL=A\otimes\omega_X^{-1}\simeq\OO_X(H).
\]
The relative right-mutation family $\cE$ has a triangle
\begin{equation}\label{eq:relativeE}
p_1^*A\otimes\pi^*\cL^\vee[1]\longrightarrow\cE\longrightarrow\cI
\longrightarrow p_1^*A\otimes\pi^*\cL^\vee[2].
\end{equation}
\end{lemma}
\begin{proof}
The bundle $A$ is acyclic. The sequence $0\to I_\Delta\to\OO_{X\times X}\to\OO_\Delta\to0$ reduces the relative Hom computation to the diagonal. Since the diagonal is a regular codimension-three embedding with normal bundle $T_X$,
\[
R\pi_*R\mathcal Hom(\OO_\Delta,p_1^*A)
\simeq (A\otimes\det T_X)[-3].
\]
The shift in the ideal sequence gives the claimed complex. The relative coevaluation map has target the dual of this complex tensored with $p_1^*A$; its fiber is a nonzero generator of $\Ext^2(I_p,A)$. Taking its shifted cone gives \eqref{eq:relativeE}.
\end{proof}

\begin{proposition}\label{prop:normal}
For the embedding $j:X\hookrightarrow N^{\mathrm{Ku}}_5$,
\[
N_{X/N^{\mathrm{Ku}}_5}\simeq K(H).
\]
\end{proposition}
\begin{proof}
The relative mutation is orthogonal to $p_1^*A$, so applying relative Hom to \eqref{eq:relativeE} identifies the self-Ext complex of $\cE$ with the relative Hom complex from $\cE$ to $\cI$. Apply $R\pi_*R\mathcal Hom(-,\cI)$ to the same triangle. Its degree-one cohomology gives an exact sequence
\begin{equation}\label{eq:normalexact}
0\longrightarrow T_X\longrightarrow j^*T_{N^{\mathrm{Ku}}_5}
\longrightarrow W\otimes\cL
\xrightarrow{\alpha}\bigwedge^2T_X.
\end{equation}
Here $\cExt^1_\pi(\cI,\cI)=T_X$ and $\cExt^2_\pi(\cI,\cI)=\bigwedge^2T_X$ are the usual universal-point Ext bundles, obtained from the diagonal Koszul resolution. Also $\pi_*(\cI\otimes p_1^*\OO_X(H))=W$. The first arrow is the Kodaira--Spencer map of the point family. The map $\alpha$ is composition with the universal coevaluation class.

Trivialize $\cL$ near a point. On a fiber, the map $\alpha$ is precisely the Yoneda map whose kernel is identified in \cite[Lemma 4.9]{APR}: a section of $I_p(H)$ is in the kernel if and only if its hyperplane section is singular at $p$. Thus its kernel is $K_p$. Changing the trivialization of $\cL$ restores the tensor factor, so the bundle kernel in \eqref{eq:normalexact} is $K\otimes\cL$. More explicitly, the restriction of $\alpha$ to this subbundle vanishes on every closed fiber, hence vanishes as a bundle morphism; both it and the kernel have constant rank three. Taking the quotient by $T_X$ proves the assertion using Lemma~\ref{lem:mutationline}.
\end{proof}

\begin{remark}\label{rem:cherncheck}
The jet sequence gives $c_1(K)=-2H$, hence $c_1(K(H))=H$ and $c_1(j^*T_N)=3H$, as required by adjunction.
\end{remark}

\begin{corollary}\label{cor:abstractblowup}
Define $\Ntilde=\Bl_{j(X)}N^{\mathrm{Ku}}_5$. Its exceptional divisor is naturally
\[
J\simeq\PP_X(K(H))\simeq\PP_X(K).
\]
Under the last identification, with $q:J\to X$,
\begin{equation}\label{eq:exceptionalnormal}
N_{J/\Ntilde}\simeq\OO_{\PP_X(K)}(-1)\otimes q^*\OO_X(H).
\end{equation}
\end{corollary}
\begin{proof}
The exceptional divisor of a blow-up along a smooth center is the projectivized normal bundle, and its normal line is the tautological subbundle. For the line convention, tensoring $K$ by $\OO_X(H)$ tensors its tautological line by the same pullback.
\end{proof}
We now identify the normalized wall graph by computing its support base ideal.

\subsection{The normalized inverse root-wall graph}\label{subsec:graph}
Let $N=N^{\mathrm{Ku}}_5$, let $\mathcal I=I_{j(X)}$, and let $\mathcal R_{\mathrm{root}}$ be the reduced closure in $M_H(w)$ of the root-sheaf locus. By \cite[Propositions~4.1 and 4.11]{APR}, the wall contraction restricts to a proper birational morphism
\[
s_{\mathrm{root}}:\mathcal R_{\mathrm{root}}\longrightarrow N
\]
which is an isomorphism over $N\setminus j(X)$. Its boundary points are the sheaves $I_{p/S}$ with $S$ singular at $p$, and $s_{\mathrm{root}}(I_{p/S})=E_p$. Denote its rational inverse by $\rho^{-1}$. We compute regular homogeneous equations for the rational support map before considering its graph.

\medskip
\noindent\textbf{The determinant section.}
The description in \cite[Section~5.1]{APR} realizes $N$ as the moduli space of stable $3$-Kronecker representations with dimension vector $(3,2)$. Choose universal multiplicity bundles $\mathcal A_3$ and $\mathcal B_2$ on $N$, of ranks three and two. Such a choice is possible because $\gcd(3,2)=1$; equivalently one may work on \'etale charts with universal representations and descend the constructions below. With $p_X,p_N$ the projections, the universal object is represented in degrees $-1,0$ by
\begin{equation}\label{eq:universalKronecker}
\boldsymbol\phi:
p_X^*\cU\otimes p_N^*\mathcal A_3
\longrightarrow
p_X^*\cQ^\vee\otimes p_N^*\mathcal B_2.
\end{equation}
Both terms have rank six. Since $\det\cU=\det\cQ^\vee=\OO_X(-H)$, put
\[
\mathcal M=(\det\mathcal B_2)^{\otimes3}
                 \otimes(\det\mathcal A_3)^{\otimes(-2)}.
\]
Taking the determinant in \eqref{eq:universalKronecker} gives
\begin{equation}\label{eq:determinantsupport}
\Delta=\det\boldsymbol\phi
\in H^0\bigl(X\times N,p_X^*\OO_X(H)\otimes p_N^*\mathcal M\bigr)
 =H^0(N,V\otimes\mathcal M).
\end{equation}
A common base-line twist of the multiplicity bundles cancels from $\mathcal M$ and this construction. In the local quotient description the same assertion follows because the determinant character is trivial on the common scalar stabilizer.

At a root object $F=i_*\OO_S(D)$, the map $\phi_F$ is injective and has cokernel $F$. At the generic point of $S$, this cokernel has length one over the discrete valuation ring of $X$. Thus $\det\phi_F$ cuts out $S$ with multiplicity one. In particular $[\Delta(F)]$ is the actual support hyperplane. At $E_p$ the cohomology of \eqref{eq:universalKronecker} instead gives, after choosing the scalar identifications,
\begin{equation}\label{eq:pointKronecker}
0\longrightarrow A\xrightarrow{i_p}\cU^{\oplus3}
\xrightarrow{\phi_p}(\cQ^\vee)^{\oplus2}
\xrightarrow{q_p}I_p\longrightarrow0,
\qquad A=\OO_X(-H).
\end{equation}
Consequently $\Delta(E_p)=0$. Since $j(X)$ is reduced, the restriction of $\Delta$ to $j(X)$ is the zero section. For a local frame of $\mathcal M$ and a basis $s_0,\ldots,s_6$ of $V$, write
\begin{equation}\label{eq:determinantcoefficients}
\Delta=\sum_{i=0}^6 g_i s_i,\qquad
\mathcal J_\Delta=(g_0,\ldots,g_6)\subset\mathcal I.
\end{equation}
These coefficient ideals glue. They are the unit ideal off $j(X)$.

\begin{lemma}\label{lem:firstordersupport}
Fix $p\in X$. With compatible choices of the one-dimensional factors, the differential of \eqref{eq:determinantsupport} at $E_p$ is the map
\begin{equation}\label{eq:determinantnormal}
d\Delta_{E_p}:\Ext^1(E_p,E_p)
\xrightarrow{\ \beta_p\ } K_p\hookrightarrow V,
\qquad K_p=H^0(I_p^2(H)),
\end{equation}
where $\beta_p$ is the Yoneda quotient associated with
$A[1]\xrightarrow{a_p}E_p\xrightarrow{e_p}I_p$:
\[
\beta_p(\xi)=e_p[1]\circ\xi\circ a_p
\in\Hom(A,I_p).
\]
In particular, $d\Delta_{E_p}$ has rank three and kernel $T_pX$. Restoring the line factor gives the normal bundle $K(H)$ of Proposition~\ref{prop:normal}.
\end{lemma}
\begin{proof}
A tangent vector $\xi$ can be represented by a first-order deformation
$\phi_p+\varepsilon\dot\phi$ of the quiver map. Indeed, the stable-representation quotient gives the tangent space as variations of the arrows modulo infinitesimal changes of bases. In the two-term complex, the composite defining $\beta_p$ is exactly
\begin{equation}\label{eq:kernelcokernel}
A\xrightarrow{i_p}\cU^{\oplus3}
\xrightarrow{\dot\phi}(\cQ^\vee)^{\oplus2}
\xrightarrow{q_p}I_p.
\end{equation}
Changing $\dot\phi$ by an infinitesimal change of bases adds terms factoring through $\phi_p$, which vanish after composition with $i_p$ or $q_p$. Thus this computation is independent of the tangent representative.

We compare \eqref{eq:kernelcokernel} with the differential of the actual determinant section. On $X\setminus\{p\}$, sequence~\eqref{eq:pointKronecker} has locally free kernel and cokernel of rank one, and $\phi_p$ has constant rank five. Locally split off its rank-five isomorphism and choose compatible frames so that
\[
\phi_p=\begin{pmatrix}I_5&0\\0&0\end{pmatrix}.
\]
The coefficient of $\varepsilon$ in $\det(\phi_p+\varepsilon\dot\phi)$ is then the bottom-right entry of $\dot\phi$. Under the determinant-line identification supplied by the exact kernel--image--cokernel sequences, this entry is the map from the kernel to the cokernel. It follows that
\begin{equation}\label{eq:determinantderivative}
d(\det)_{\phi_p}(\dot\phi)
=\bigl[A\xrightarrow{\ q_p\dot\phi i_p\ }I_p
                \hookrightarrow\OO_X\bigr]
\quad\text{in }H^0(X,\OO_X(H)),
\end{equation}
up to the fixed nonzero scalar from the line identifications. To justify the equality globally, the determinant-line identification on $X\setminus\{p\}$ extends across $p$: both sides are line bundles on the smooth variety $X$, and the omitted set has codimension three. Any resulting automorphism of $\OO_X(H)$ is a scalar. The two global sections therefore agree once that scalar is fixed; no rank assumption at $p$ itself is needed.

Finally, \cite[Lemma~4.9]{APR}, equivalently the fiber of \eqref{eq:normalexact}, gives
\[
0\longrightarrow\Ext^1(I_p,I_p)\longrightarrow\Ext^1(E_p,E_p)
\xrightarrow{\beta_p}
\ker\!\left(H^0(I_p(H))\xrightarrow{\alpha_p}\Ext^2(I_p,I_p)\right)
\longrightarrow0.
\]
Here the left map is the derivative of $j$, and the kernel on the right consists of the sections whose first jet at $p$ vanishes, namely $K_p$. Combining this exact sequence with \eqref{eq:determinantderivative} proves \eqref{eq:determinantnormal}. This argument applies to every first-order tangent vector and does not presuppose any transverse one-parameter family.
\end{proof}

\begin{lemma}\label{lem:graphideal}
The coefficient ideal of the determinant support section is exactly the reduced point-locus ideal:
\[
\mathcal J_\Delta=I_{j(X)}.
\]
Thus on a sufficiently small trivializing chart the support map has homogeneous coordinates $g_0,\ldots,g_6$ with $(g_0,\ldots,g_6)=I_{j(X)}$.
\end{lemma}
\begin{proof}
The inclusion $\mathcal J_\Delta\subset\mathcal I$ was established in \eqref{eq:determinantcoefficients} using an actual regular determinant section. Since that section vanishes along $j(X)$, its coefficient classes define a conormal map
\[
(V^\vee\otimes\mathcal M^\vee)|_{j(X)}
\longrightarrow\mathcal I/\mathcal I^2.
\]
On every fiber its dual is the injective normal differential in Lemma~\ref{lem:firstordersupport}. Hence this conormal map is surjective. At the local ring $R=\OO_{N,E_p}$, writing $I=\mathcal I_{E_p}$ and $J_\Delta=(\mathcal J_\Delta)_{E_p}$, we obtain
\[
J_\Delta+I^2=I.
\]
For the finite $R$-module $Q=I/J_\Delta$, this says $Q=IQ$. As $I$ is contained in the maximal ideal of $R$, Nakayama's lemma gives $Q=0$. The two ideals therefore agree at every point of $j(X)$; off $j(X)$ both are the unit ideal. This proves equality of ideal sheaves, rather than only equality of their radicals or of their projectivized normal directions.
\end{proof}

\begin{theorem}\label{thm:graphblowup}
Let $\Gamma_{\rho^{-1}}$ be the reduced graph closure of the inverse principal root-wall map. Then
\[
\operatorname{Norm}(\Gamma_{\rho^{-1}})
\simeq\Bl_{j(X)}N^{\mathrm{Ku}}_5=\Ntilde.
\]
Its exceptional divisor is $J\simeq\PP_X(K(H))\simeq\PP_X(K)$, and $(p,[s])\in J$ records the limiting boundary root sheaf $I_{p/S_s}$.
\end{theorem}
\begin{proof}
Let $G$ be the reduced closure of the graph of $[\Delta]:N\setminus j(X)\to\PP(V)=|H|$. On a trivializing chart, the closure of a projective graph is the Proj of the Rees algebra of its coefficient ideal. Lemma~\ref{lem:graphideal} therefore gives
\[
G\simeq\Proj_N\bigoplus_{m\ge0}\mathcal I^m
=\Bl_{j(X)}N.
\]
The local isomorphisms glue: changing a frame of $\mathcal M$ multiplies all coordinates by a unit, and each isomorphism restricts to the same graph over $N\setminus j(X)$. Since both $N$ and $j(X)$ are smooth, $G$ is smooth and normal.

We also need a regular support morphism on the entire reduced root component. Here is a family-level construction. Let $\mathcal G$ be the universal Gieseker sheaf pulled back to $X\times\mathcal R_{\mathrm{root}}$, and write $\pi_R$ for the projection. Every fiber is a rank-one torsion-free sheaf on its integral normal hyperplane support $S$, by \cite[Lemma~3.2 and Proposition~4.1]{APR}. Thus
\[
\mathcal{H}om_{\OO_S}(\mathcal G_r,\mathcal G_r)\simeq\OO_S,
\qquad
\Hom_X(\mathcal G_r,\mathcal G_r(H))\simeq H^0(S,\OO_S(H)).
\]
The latter space has dimension six, and multiplication by sections gives a surjection
$V\to H^0(S,\OO_S(H))$ whose kernel is the support-equation line. The perfect relative Hom complex has no negative fiber cohomology. Since its degree-zero fiber dimension is constant and the base $\mathcal R_{\mathrm{root}}$ is reduced, its degree-zero cohomology is locally free and commutes with base change. Explicitly, a local nonnegative complex representing it has a first differential of constant rank; its larger minors vanish on the reduced base, so its kernel is a subbundle. Consequently the multiplication map
\[
V\otimes\OO_{\mathcal R_{\mathrm{root}}}
\longrightarrow
\mathcal H^0\!\left(R\pi_{R*}\RHom(\mathcal G,\mathcal G(H))\right)
\]
is a surjection of vector bundles with a line-bundle kernel. This kernel defines a regular support morphism
$h:\mathcal R_{\mathrm{root}}\to\PP(V)$. The construction is unchanged by a base-line twist of $\mathcal G$.

The pair consisting of the projection to $N$ and $h$ induces a proper morphism
\[
\Gamma_{\rho^{-1}}\longrightarrow G.
\]
Its image factors through $G$ because the source is reduced and the asserted factorization holds on the dense root locus. It is an isomorphism over $N\setminus j(X)$. By the APR boundary classification, over a fixed pair $(E_p,[s])$ there is exactly one possible geometric point, represented by $I_{p/S_s}$. Thus its geometric fibers are finite, so it is quasi-finite, hence finite by properness. Normalization is finite for varieties over $\CC$, and we obtain a finite birational morphism
\[
\operatorname{Norm}(\Gamma_{\rho^{-1}})\longrightarrow G.
\]
It is an isomorphism because $G$ is normal. The exceptional divisor and its support label follow from Lemma~\ref{lem:firstordersupport}, Proposition~\ref{prop:normal}, and Corollary~\ref{cor:abstractblowup}.
\end{proof}

\begin{remark}\label{rem:contactorder}
For an arc with $\gamma^*\mathcal I=(t^e)$, Lemma~\ref{lem:graphideal} writes $\gamma^*\Delta=t^e(s_0+O(t))$ with $0\ne s_0\in K_p$. Its limiting support is $[s_0]$; the first-order normal class detects this line only for $e=1$.

For later use, let $\pi:\Ntilde\to N$ be the blow-up morphism. Its differential does give an injective map on the normal line at each $b=(p,[s])\in J$:
\begin{equation}\label{eq:blowupnormaldifferential}
N_{J/\Ntilde,b}\hookrightarrow N_{X/N,p},
\qquad\im(d\pi)=\ell_b,
\end{equation}
where $\ell_b$ is the tautological line corresponding to $[s]\in\PP(K_p)$, with the twist in \eqref{eq:exceptionalnormal}. To see this directly, choose local normal coordinates $x_1,x_2,x_3$ for $j(X)$. On a blow-up chart,
\[
x_1=t,\qquad x_2=tu_2,\qquad x_3=tu_3,\qquad J=(t=0).
\]
Modulo $T_pX$, the differential sends $\partial_t$ to $(1,u_2,u_3)$, which is nonzero and spans $\ell_b$. The other charts give the same intrinsic tautological inclusion. This nonvanishing follows from the proved blow-up description, not from an assumption about the order of a test family.
\end{remark}

\section{The determinantal incidence variety}\label{sec:det}
This section establishes the geometry of the determinantal scheme $\Sigma$ in Theorem~\ref{thm:mainambient}(iii). Its identification with the actual scheme-theoretic moduli intersection is proved in Section~\ref{sec:ambient}.

The first-jet sequence gives $W/K\simeq\Omega_X^1(H)$, of rank three. On the smooth elevenfold
\[
Y=\Gr_X(2,W)
\]
let $\cS$ be the tautological rank-two bundle and $\pi_Y:Y\to X$ the projection. Define
\begin{equation}\label{eq:Sigma}
\Sigma=D_1\bigl(\cS\to\pi_Y^*(W/K)\bigr),\qquad Z=\Gr_X(2,K),
\end{equation}
using the ideal of all $2\times2$ minors. The further embedding $Y\hookrightarrow X\times\Gr(2,V)$ is the zero scheme of the universal evaluation map $\cS\to\OO_X(H)$.

\begin{proposition}\label{prop:Sigma}
The scheme $\Sigma$ is irreducible, normal, Cohen--Macaulay, and of dimension nine. Its reduced singular locus is $Z$, of dimension five. Its incidence resolution is
\begin{equation}\label{eq:incidence}
\widetilde\Sigma=\{(p,\ell,L):\ell\subset K_p,\ \ell\subset L\subset W_p\}
\longrightarrow\Sigma.
\end{equation}
It is a $\PP^4$-bundle over $\PP_X(K)$, and has fiber $\PP^1$ over $Z$ and a single reduced point elsewhere.
\end{proposition}
\begin{proof}
Locally split $W=K\oplus Q$, with $Q$ of rank three. Near a point $L\subset K_p$, a Grassmannian chart writes a moving two-plane as a graph into $(K_p/L)\oplus Q_p$. The two coordinates into $K_p/L$ are free, while the six coordinates into $Q_p$ form a generic $3\times2$ matrix. The equations of \eqref{eq:Sigma} are exactly its maximal minors. The base point contributes three more free coordinates. Thus, transversely to $Z$, the model is $D_1(\Mat_{3\times2})$, the affine cone over $\PP^2\times\PP^1$.

At a rank-one point the rank condition is smooth of codimension two. At a rank-zero point the generic determinantal model just described has dimension four and singular locus its vertex. The rank-one generic determinantal ring is a normal Cohen--Macaulay domain (equivalently, the standard Segre ring). These charts prove the local assertions, and show that $\dim\Sigma=11-2=9$ and $\Sing(\Sigma)_{\mathrm{red}}=Z$.

In \eqref{eq:incidence}, after choosing $\ell\subset K_p$, the two-plane $L$ is determined by a line in $W_p/\ell$, a five-dimensional vector space. This makes the incidence variety smooth and irreducible of dimension $3+2+4=9$. Its proper image is the whole underlying locus of \eqref{eq:Sigma}, so $\Sigma$ is irreducible. Away from $Z$ its inverse is the kernel line of the rank-one map; over $L\subset K_p$ any line of $L$ is allowed. This proves the fiber statement and the resolution assertion.
\end{proof}

\begin{remark}\label{lem:artin}
Incidence lifting does not characterize the determinantal functor on Artin bases. Over $D=\CC[\varepsilon]/(\varepsilon^2)$, the matrix with columns $(\varepsilon,0,0)^t$ and $(0,\varepsilon,0)^t$ has zero maximal minors but kernel $\varepsilon D^2$, which contains no unimodular vector. Its $D$-point does not lift to the incidence resolution. We compute the moduli intersection from the component ideals in Section~\ref{sec:ambient}.
\end{remark}

\section{Extension spaces, base change, and the relative family}\label{sec:extensions}
This section globalizes the degree-five extension calculation over the normalized inverse-wall graph. The resulting rank-five bundle and its boundary base-change formula are the inputs used in Section~\ref{sec:ambient} to construct the actual principal-component morphism.

\subsection{Boundary Ext dimensions}
For a hyperplane section $S=(s=0)$ containing $p$, put $T_{p,S}=i_*I_{p/S}(-H)$.
\begin{lemma}\label{lem:boundaryext}
There are identifications of vector spaces
\[
\Ext^1_X(T_{p,S},A)\simeq H^0(S,\OO_S(H))\simeq V/\langle s\rangle,
\qquad \dim\Ext^2_X(T_{p,S},A)=1.
\]
Thus the dimensions are $(6,1)$, while $\Ext^0$ and $\Ext^3$ vanish.
\end{lemma}
\begin{proof}
Cancel the common twist and apply $R\Hom_X(-,\OO_X)$ to
$0\to i_*I_{p/S}\to i_*\OO_S\to\OO_p\to0$.
For the Cartier divisor, $R\mathcal Hom_X(i_*\OO_S,\OO_X)=i_*\OO_S(H)[-1]$; for the point the only nonzero Ext into $\OO_X$ is in degree three and is one-dimensional. The vanishings $H^1(S,\OO_S(H))=H^2(S,\OO_S(H))=0$ follow from $0\to\OO_X\to\OO_X(H)\to\OO_S(H)\to0$ and Kodaira vanishing on $X$. The resulting long exact sequence proves the claim, and identifies the degree-two term with $\Ext^3_X(\OO_p,\OO_X)$.
\end{proof}

\begin{lemma}\label{lem:simplejump}
Let $R$ be a regular local ring and let $t$ define a smooth divisor. Suppose a perfect complex $\mathcal C$ of amplitude $[1,2]$ has derived-fiber dimensions $(5,0)$ off the divisor and $(6,1)$ on it. Suppose, at the point under consideration, its first normal differential is a nonzero map from the degree-one fiber to the degree-two fiber. Then, after shrinking,
\[
\mathcal C\simeq[\,R^5\oplus R\xrightarrow{(0,\ldots,0,t)}R\,]
\quad\text{in degrees }1,2.
\]
In particular $\mathcal H^1(\mathcal C)$ is free of rank five and its special-fiber base-change image is the kernel of the first normal differential.
\end{lemma}
\begin{proof}
Split off invertible blocks from a two-term free representative. At the chosen point the remaining ranks are six and one. Its row vanishes along $t=0$, hence has the form $t(a_1,\ldots,a_6)$. The hypothesis says that some $a_i$ is a unit. A change of basis reduces the row to $(0,\ldots,0,t)$. The kernel and the derived special fiber are then computed directly.
\end{proof}
The nonzero-differential assumption is substantive. The rows $(0,\ldots,0,t)$ and $(0,\ldots,0,t^2)$ have the same ordinary fiber dimensions, but the latter has zero first normal differential. Lemma~\ref{lem:boundaryext} does not compute this differential for a geometric family.

\subsection{The global residual family and the boundary obstruction map}\label{subsec:relativeext}
The normalized graph theorem supplies a morphism
\[
r:\Ntilde\longrightarrow \mathcal R_{\mathrm{root}}\subset M_H(w)
\]
to the Gieseker root component. The space $M_H(w)$ admits a universal family by \cite[Section~4.1]{APR}. Fix one such family $\mathcal G$ and put
\begin{equation}\label{eq:residualfamily}
\cT=(\id_X\times r)^*\mathcal G\otimes p_X^*\OO_X(-H)
\qquad\text{on }X\times\Ntilde.
\end{equation}
Changing the normalization of $\mathcal G$ tensors $\cT$ by the pullback of a line bundle on $\Ntilde$; every projectivized construction below is therefore independent of this choice.

Let
\begin{equation}\label{eq:relativeextcomplex}
\mathcal C=R\pi_*R\mathcal Hom(\cT,p_X^*A),\qquad A=\OO_X(-H).
\end{equation}
This is a perfect complex. Away from $J$, a fiber is a root sheaf $i_*\OO_S(D-H)$. By Grothendieck duality its relevant Ext groups are
\[
\Ext_X^k(i_*\OO_S(D-H),A)\simeq H^{k-1}(S,\OO_S(H-D)).
\]

\begin{lemma}\label{lem:rootcohomology}
Let $S\in|H|$ carry a root $D$ as in the APR root locus. The rank-one reflexive sheaf $\OO_S(H-D)$ satisfies
\[
H^i(S,\OO_S(H-D))=0\quad(i>0),\qquad
h^0(S,\OO_S(H-D))=5.
\]
Consequently
\begin{equation}\label{eq:rootextglobal}
\dim\Ext_X^1(i_*\OO_S(D-H),A)=5,\qquad \Ext_X^2(i_*\OO_S(D-H),A)=0.
\end{equation}
\end{lemma}
\begin{proof}
By \cite[Lemma~3.2 and its proof]{APR}, a singular $S$ has rational double points or is an elliptic cone. The latter has no roots, so it is excluded by the existence of $D$. Thus $S$ is smooth or has rational double points. The divisor $-D$ is again a root, and \cite[Lemma~4.17]{APR} gives
\[
H^i(S,\OO_S(-D))=0\qquad\text{for all }i.
\]
Choose a general member $E\in|\OO_S(H)|$. Very ampleness and Bertini give a smooth connected curve disjoint from the finite singular locus of $S$. Adjunction, with $K_S=-H|_S$, gives $\omega_E\simeq\OO_E$, so $E$ has genus one. The sheaf $\OO_S(-D)$ is locally free along $E$, and multiplication by its defining section gives
\[
0\longrightarrow\OO_S(-D)\longrightarrow\OO_S(H-D)
\longrightarrow\OO_S(H-D)|_E\longrightarrow0.
\]
The last term is a line bundle on $E$ of degree $(H-D)\cdot H=5$. It has $h^0=5$ and $h^1=0$. The long exact cohomology sequence proves the claimed vanishing and dimension. This argument requires no Cartier assumption on $D$ at the singular points.

For the Ext statement, a rank-one reflexive sheaf on the normal Gorenstein surface $S$ is maximal Cohen--Macaulay. Its higher sheaf Ext groups into $\OO_S$ vanish, and its dual is $\OO_S(-D)$. Grothendieck duality for the Cartier embedding $i:S\hookrightarrow X$ therefore gives
\[
R\mathcal Hom_X(i_*\OO_S(D-H),A)
\simeq i_*\OO_S(H-D)[-1].
\]
Taking cohomology proves \eqref{eq:rootextglobal}.
\end{proof}

Thus \eqref{eq:rootextglobal} holds throughout $\Ntilde\setminus J$. On $J$ the fibers are $T_{p,S}=i_*I_{p/S}(-H)$ and Lemma~\ref{lem:boundaryext} gives dimensions $(6,1)$. Thus $\mathcal C$ has relative amplitude $[1,2]$ near $J$.

\begin{proposition}\label{prop:evaluation}
Let $b=(p,[s])\in J$, with $S=(s=0)$ singular at $p$. Under the natural identifications
\[
\Ext_X^1(T_{p,S},A)\simeq V/\langle s\rangle,
\qquad
\Ext_X^2(T_{p,S},A)\simeq \OO_X(H)|_p\otimes L_b
\]
for a one-dimensional normalization factor $L_b$, the first differential of \eqref{eq:relativeextcomplex} in a nonzero normal direction to $J$ is, up to the choice of a generator of $L_b$,
\begin{equation}\label{eq:evaluation}
V/\langle s\rangle\longrightarrow \OO_X(H)|_p,
\qquad [g]\longmapsto g(p).
\end{equation}
In particular its kernel is $W_p/\langle s\rangle=H^0(I_p(H))/\langle s\rangle$.
\end{proposition}
\begin{proof}
Twist by $H$ and write $F=I_{p/S}$ and $L=\OO_X(-H)$. The exact sequence
\[
0\longrightarrow L\xrightarrow{s}I_p\longrightarrow F\longrightarrow0
\]
defines $a:F\to L[1]$. For $g\in V=\Hom(L,\OO_X)$ the class
\[
e_g=g[1]\circ a\in\Ext_X^1(F,\OO_X)
\]
depends only on $[g]\in V/\langle s\rangle$, and these classes give the natural identification with that quotient.

Let $n$ be a nonzero normal direction to $J$ and let $\kappa_n\in\Ext^1(F,F)$ be the Kodaira--Spencer class of the family \eqref{eq:residualfamily}. The standard obstruction to deforming the extension $e_g$ is the Yoneda product
\[
\delta_b(e_g)=e_g[1]\circ\kappa_n
=g[2]\circ\eta_n,
\qquad
\eta_n:=a[1]\circ\kappa_n\in\Ext_X^2(F,L).
\]
The target $\Ext_X^2(F,L)$ is one-dimensional. Indeed, Grothendieck duality and $0\to I_{p/S}\to\OO_S\to\OO_p\to0$ identify it with $\Ext_S^2(\OO_p,\OO_S)$.

We claim $\eta_n\ne0$. The obstruction to lifting the extension class $a\in\Ext^1(F,L)$ over the first-order deformation $F_\varepsilon$ is precisely $a[1]\circ\kappa_n$. The line bundle $L$ is kept fixed. If this class vanished, a lift would give a flat extension: its middle term is a deformation of $I_p$, since both endpoints are flat and its closed fiber is the original extension. Thus the defining sequence would lift to
\[
0\longrightarrow L_\varepsilon\longrightarrow I_\varepsilon\longrightarrow F_\varepsilon\longrightarrow0.
\]
Since $\Hilb^1(X)=X$, the deformation $I_\varepsilon$ is the ideal of an infinitesimal point $p_\varepsilon$, so $F_\varepsilon$ remains in the point-on-hyperplane component of $M_H(w)$. The APR contraction sends that component infinitesimally into $T_pX$. On the other hand, let $r:\Ntilde\to\mathcal R_{\mathrm{root}}$ be the morphism from the normalized graph. Theorem~\ref{thm:graphblowup} identifies $s_{\mathrm{root}}\circ r$ with the blow-up morphism $\pi:\Ntilde\to N^{\mathrm{Ku}}_5$. By \eqref{eq:blowupnormaldifferential}, the nonzero normal class $n$ has nonzero image under $d\pi$ in
\[
N_{X/N^{\mathrm{Ku}}_5,p}\simeq K_p\otimes\OO_X(H)|_p,
\]
spanning the tautological line labelled by $[s]$. This contradicts the preceding infinitesimal factorization through $X$. Hence $\eta_n$ generates the one-dimensional group.

Finally, under local duality the morphism $g:L\to\OO_X$ acts on the point-supported generator by multiplication by the fiber value $g(p)$. Therefore $g[2](\eta_n)=g(p)\eta'_n$ for a nonzero generator $\eta'_n$ of the target obstruction line, which proves \eqref{eq:evaluation}.
\end{proof}

\begin{theorem}\label{thm:extbundle}
Set
\begin{equation}\label{eq:E5}
\cE_5:=\mathcal H^1(\mathcal C).
\end{equation}
Then $\cE_5$ is locally free of rank five on $\Ntilde$. Replacing $\mathcal G$ by $\mathcal G\otimes\pi^*M$ replaces $\cE_5$ by $\cE_5\otimes M^\vee$, so its projectivization is canonical. On $J=\PP_X(K)$, with $q:J\to X$ and tautological line $\mathcal L_J\subset q^*K$, there is a line bundle $\mathcal B_J$ specified below such that
\begin{equation}\label{eq:boundarybundle}
\cE_5|_J\simeq(q^*W/\mathcal L_J)\otimes\mathcal B_J.
\end{equation}
In particular $\PP_J(\cE_5|_J)\simeq\PP_J(q^*W/\mathcal L_J)$.
\end{theorem}
\begin{proof}
Off $J$, \eqref{eq:rootextglobal} and cohomology and base change give rank five. On $J$, Lemma~\ref{lem:boundaryext} gives derived-fiber dimensions $(6,1)$, and Proposition~\ref{prop:evaluation} gives a nonzero normal differential. Lemma~\ref{lem:simplejump} puts $\mathcal C$ locally in the form $[R^5\oplus R\xrightarrow{(0,\ldots,0,t)}R]$ in degrees $1,2$. Its degree-one cohomology is therefore locally free of rank five.

We globalize the boundary identification. Write $\pi_J:X\times J\to J$ and let $\Gamma_q\subset X\times J$ be the graph of $q$. The tautological equation gives the flat boundary family
\begin{equation}\label{eq:universalboundarysequence}
0\to p_X^*\OO_X(-H)\otimes\pi_J^*\mathcal L_J
\to I_{\Gamma_q}\to\mathcal G_J^0\to0.
\end{equation}
The restriction of the chosen universal family is $\mathcal G|_{X\times J}\simeq\mathcal G_J^0\otimes\pi_J^*M_J$ for a line bundle $M_J$ on $J$. Indeed, both families classify the same stable sheaves; their relative Isom torsor is the scalar $\mathbb G_m$-torsor of such a line bundle. Put
$\mathcal B_J=\mathcal L_J^\vee\otimes M_J^\vee$.

For $i:J\hookrightarrow\Ntilde$, let $B_k=\mathcal H^k(Li^*\mathcal C)$. Relative Hom applied to \eqref{eq:universalboundarysequence}, with the common $-H$ twist in $\mathcal C$ cancelled, gives
\[
B_1\simeq(q^*V/\mathcal L_J)\otimes\mathcal B_J,
\qquad B_2\simeq q^*\det T_X\otimes M_J^\vee.
\]
Here $q^*V$ denotes $V\otimes\OO_J$. The first formula is the cokernel of $\OO_J\to q^*V\otimes\mathcal L_J^\vee$, tensored by $M_J^\vee$; the second follows from point duality $\cExt^3_{\pi_J}(\OO_{\Gamma_q},\OO)=q^*\det T_X$.

Dividing a local differential of $\mathcal C$ by a local equation of $J$ defines the intrinsic first-normal morphism
\[
\partial_J:B_1\longrightarrow B_2\otimes N_{J/\Ntilde}^\vee.
\]
Changing the equation changes this division by the transition function of $N_{J/\Ntilde}^\vee$, so the morphisms glue. By \eqref{eq:exceptionalnormal} and $\det T_X=\OO_X(2H)$ its target is $q^*\OO_X(H)\otimes\mathcal B_J$. Proposition~\ref{prop:evaluation} says that on each closed fiber $\partial_J$ is a nonzero scalar multiple of evaluation. Its restriction to $(q^*W/\mathcal L_J)\otimes\mathcal B_J$ therefore vanishes: this is a morphism of bundles on the reduced scheme $J$. The induced map on the quotient line bundle is everywhere nonzero, hence an isomorphism. Equivalently, after this target-line identification, the global sequence is
\begin{equation}\label{eq:relativeboundarykernel}
0\to(q^*W/\mathcal L_J)\otimes\mathcal B_J
\to B_1\xrightarrow{\partial_J}q^*\OO_X(H)\otimes\mathcal B_J\to0.
\end{equation}
The simple-jump form identifies the natural base-change map $\cE_5|_J\to B_1$ with the inclusion $\ker\partial_J$. These maps are intrinsic and glue, proving \eqref{eq:boundarybundle} as a bundle identity, rather than only fiberwise. Under $\mathcal G\mapsto\mathcal G\otimes\pi^*M$, one has $M_J\mapsto M_J\otimes M|_J$, so $\mathcal B_J\mapsto\mathcal B_J\otimes M^\vee|_J$, as required.
\end{proof}

\begin{corollary}\label{cor:principalextension}
The projective bundle
\begin{equation}\label{eq:Rtilde}
\Rtilde:=\PP_{\Ntilde}(\cE_5)
\end{equation}
is canonical, smooth, irreducible, and of dimension ten. Over the open root locus it is the $\PP^4$-family of nonzero extensions by $\OO_X(-H)$, while on the exceptional divisor
\[
\Rtilde|_J\simeq\PP_J(q^*W/\mathcal L_J)
\simeq\widetilde\Sigma
=\{(p,\ell,L):\ell\subset K_p,\ \ell\subset L\subset W_p\}.
\]
\end{corollary}
\begin{proof}
Smoothness and the dimension statement follow because $\Ntilde$ is the blow-up of a smooth sixfold along a smooth center and $\cE_5$ has rank five. The open-locus interpretation is the usual extension construction. Projectivizing \eqref{eq:boundarybundle} removes $\mathcal B_J$: a line in $(q^*W/\mathcal L_J)_{(p,[s])}=W_p/\langle s\rangle$ determines the two-plane $L=\langle s,g\rangle$ with $\langle s\rangle\subset L\subset W_p$, and conversely.
\end{proof}

\section{The principal component and the local structure of the moduli space}\label{sec:ambient}
Using the APR chamber comparison, the normalized inverse-wall graph, and the extension bundle of Theorems~\ref{thm:exactprincipal}, \ref{thm:graphblowup}, and \ref{thm:extbundle}, we construct the principal moduli morphism and compute its formal contraction to the normalization of its image. We then determine the complete equivariant local deformation ring at every corrected polystable point. This calculation proves that the principal component is normal and identifies its scheme-theoretic intersection with the corrected component, establishing Theorem~\ref{thm:mainambient} and completing Theorem~\ref{thm:mainconstruction}(ii).

Fix from now on a rational outer parameter
\[
\sigma=\sigma_{\sqrt{b^2-2},-b,s},\qquad b>\frac32,\quad s>0,
\]
with $b,s\in\mathbb Q$. Since only $\alpha^2=b^2-2$ enters the central charge, the existence input in Section~\ref{sec:hearts} applies and gives the proper good moduli space
\[
M^{ss}=M_\sigma(-v_5).
\]

\subsection{The principal morphism and its fibers}
\begin{lemma}\label{lem:offboundarystable}
Let $x\in\Rtilde\setminus\widetilde\Sigma$ lie over a root sheaf
\[
T_x=i_*\OO_S(D-H)
\]
on an integral normal hyperplane section $S$. Every nonzero extension represented by $x$ has middle term $I_C$ for a Cohen--Macaulay curve $C\subset X$ with Hilbert polynomial $5m+1$. Moreover $I_C[1]$ is $\sigma$-stable.
\end{lemma}
\begin{proof}
By Grothendieck duality, the extension space is naturally
\[
\Ext_X^1(i_*\OO_S(D-H),\OO_X(-H))
 \simeq H^0(S,\OO_S(H-D)).
\]
Thus a nonzero extension class is represented by a nonzero map
\[
u:\OO_S(D-H)\longrightarrow\OO_S.
\]
The surface $S$ is integral and $\OO_S(D-H)$ is rank-one reflexive, hence $u$ is injective. Write its image as an ideal $I_{C/S}\subset\OO_S$. Pulling this ideal back through
\[
0\longrightarrow\OO_X(-H)\longrightarrow\OO_X\longrightarrow i_*\OO_S\longrightarrow0
\]
identifies the extension middle term with $I_C$ and gives
\[
0\longrightarrow\OO_X(-H)\longrightarrow I_C
 \longrightarrow i_*I_{C/S}\longrightarrow0.
\]
The rank-one reflexive sheaf $I_{C/S}\simeq\OO_S(D-H)$ has depth two on the normal Cohen--Macaulay surface $S$. The depth lemma applied to
\[
0\longrightarrow I_{C/S}\longrightarrow\OO_S\longrightarrow\OO_C\longrightarrow0
\]
shows that $C$ is Cohen--Macaulay of pure dimension one. Its numerical class is fixed by the extension, so $\ch(I_C)=v_5$, equivalently $P_C(m)=5m+1$.

The rank-one torsion-free sheaf $I_C$ is slope-stable. Starting at large $\alpha$ with $\beta=-b$ and moving vertically to $\alpha=\sqrt{b^2-2}$, any actual tilt wall for the class $v_5$ would have to strictly enclose the standard terminal wall $W_5$. Proposition~\ref{prop:nolarger} excludes such a wall. Hence $I_C$ is tilt-stable of slope zero at the chosen outer parameter.

Since $C$ is Cohen--Macaulay, $I_C$ has depth at least two at every closed point of the smooth threefold $X$. Therefore
\[
\Ext_X^1(\OO_p,I_C)=0
\qquad\text{for all }p\in X.
\]
Lemma~\ref{lem:stablecriterion} now gives the $\sigma$-stability of $I_C[1]$.
\end{proof}

\begin{lemma}\label{lem:boundaryelementary}
Let $(p,\ell,L)\in\widetilde\Sigma$, with $\ell=\langle s\rangle\subset K_p$ and $L=\langle s,g\rangle\subset W_p$. Let $F_{p,\ell,L}$ be the middle term of the corresponding boundary extension
\[
0\longrightarrow\OO_X(-H)\longrightarrow F_{p,\ell,L}
 \longrightarrow i_*I_{p/S_s}(-H)\longrightarrow0.
\]
Then there is a canonical elementary-modification sequence
\begin{equation}\label{eq:boundaryelementary}
0\longrightarrow F_{p,\ell,L}
 \longrightarrow I_{E_L}\longrightarrow\OO_p\longrightarrow0.
\end{equation}
Consequently, in the second-tilt heart,
\begin{equation}\label{eq:boundaryheart}
0\longrightarrow\OO_p\longrightarrow F_{p,\ell,L}[1]
 \longrightarrow I_{E_L}[1]\longrightarrow0.
\end{equation}
In particular $F_{p,\ell,L}[1]$ is $\sigma$-semistable and
\[
\operatorname{gr}\bigl(F_{p,\ell,L}[1]\bigr)
 \simeq \OO_p\oplus I_{E_L}[1].
\]
Its S-equivalence class depends only on $(p,L)$ and not on $\ell$.
\end{lemma}
\begin{proof}
Twist by $H$. The residual boundary sheaf fits into
\[
0\longrightarrow\OO_X(-H)\xrightarrow{s} I_p
 \longrightarrow i_*I_{p/S_s}\longrightarrow0.
\]
Proposition~\ref{prop:evaluation} identifies the allowed boundary extension class with the class of $g$ modulo $\langle s\rangle$, where $g(p)=0$. The pushout presentation of the twisted middle term is therefore
\begin{equation}\label{eq:pushoutpresentation}
0\longrightarrow\OO_X(-H)
 \xrightarrow{(s,-g)} I_p\oplus\OO_X
 \longrightarrow F_{p,\ell,L}(H)\longrightarrow0.
\end{equation}
On the other hand, the two sections $s,g$ cut out the codimension-two linear section $E_L$, and the Koszul presentation is
\[
0\longrightarrow\OO_X(-H)
 \xrightarrow{(s,-g)}\OO_X^{\oplus2}
 \longrightarrow I_{E_L}(H)\longrightarrow0.
\]
The inclusion $I_p\oplus\OO_X\hookrightarrow\OO_X^{\oplus2}$ and the snake lemma yield
\[
0\longrightarrow F_{p,\ell,L}(H)
 \longrightarrow I_{E_L}(H)\longrightarrow\OO_p\longrightarrow0,
\]
which is \eqref{eq:boundaryelementary} after twisting back.

Both $F_{p,\ell,L}$ and $I_{E_L}$ have first three Chern-character entries $(1,0,-1)$ and lie in the nonpositive part of the first-tilt torsion pair at the outer point. Thus the shifted triangle associated with \eqref{eq:boundaryelementary} is the short exact sequence \eqref{eq:boundaryheart} in $\cA_{\alpha,\beta}$. Proposition~\ref{prop:CI} says that $I_{E_L}[1]$ is stable, and $\OO_p$ is a stable phase-one skyscraper object. Hence their extension is semistable of phase one and has the displayed graded object. The graded object records only $p$ and $L$, proving the last assertion.
\end{proof}

\begin{theorem}\label{thm:principalcomponent}
The tautological relative extensions on $\Rtilde=\PP_{\Ntilde}(\cE_5)$ define a morphism
\[
\Phi:\Rtilde\longrightarrow M^{ss}.
\]
Its image is the reduced closure $R$ of the smooth rational-quintic locus, and
\[
\Phi:\Rtilde\longrightarrow R
\]
is proper and birational. Moreover:
\begin{enumerate}
\item $\Phi$ is an isomorphism from $\Rtilde\setminus\widetilde\Sigma$ onto its image;
\item on the boundary one has
\[
\Phi|_{\widetilde\Sigma}=\iota_\sigma\circ\eta,
\qquad
\eta:\widetilde\Sigma\to\Sigma,
\quad (p,\ell,L)\mapsto(p,L),
\]
where $\iota_\sigma(p,L)$ is the S-equivalence class of $I_{E_L}[1]\oplus\OO_p$;
\item distinct pairs $(p,L)$ define distinct S-equivalence classes, and the only positive-dimensional geometric fibers of $\Phi$ are the fibers
\[
\eta^{-1}(p,L)\simeq\PP^1
\qquad ((p,L)\in Z=\Gr_X(2,K)).
\]
\end{enumerate}
\end{theorem}
\begin{proof}
The tautological line
\[
\OO_{\Rtilde}(-1)\hookrightarrow q^*\cE_5
\]
and the definition $\cE_5=\mathcal H^1(R\pi_*R\mathcal Hom(\cT,p_X^*A))$ give, by the standard relative-Ext construction, a universal extension on $X\times\Rtilde$, unique up to tensoring by a line bundle pulled back from $\Rtilde$. The simple-jump normal form of Theorem~\ref{thm:extbundle} ensures that on $J$ its fiberwise extension classes are exactly the base-change kernel $W_p/\langle s\rangle$. Lemmas~\ref{lem:offboundarystable} and \ref{lem:boundaryelementary} show that every geometric fiber, after shifting by one, is $\sigma$-semistable of class $-v_5$. The universal extension is a flat family (equivalently, its cone has relative Tor-amplitude zero, since all fibers are sheaves with the fixed Hilbert polynomial). It therefore induces a morphism to the semistable stack and hence, by the universal property of the good moduli space, a morphism
\[
\Phi:\Rtilde\to M^{ss}.
\]

We first identify the morphism away from the boundary. For an off-boundary point write the resulting object as $I_C[1]$ and the residual root as $i_*\OO_S(D-H)$. Twisting the defining extension by $H$ gives
\begin{equation}\label{eq:recoverhyperplane}
0\longrightarrow\OO_X\longrightarrow I_C(H)
 \longrightarrow i_*\OO_S(D)\longrightarrow0.
\end{equation}

For a root $D$, the vanishing $H^\bullet(S,\OO_S(D))=0$ is \cite[Lemma~4.17]{APR}; this includes the singular root surfaces and does not require treating a Weil divisor as Cartier. Thus \eqref{eq:recoverhyperplane} yields
\[
H^0(X,I_C(H))\simeq\CC,
\qquad H^i(X,I_C(H))=0\quad(i>0).
\]

Therefore the unique section of $I_C(H)$ recovers the hyperplane $S$. Its quotient in \eqref{eq:recoverhyperplane} recovers the residual root sheaf, hence the point of $\Ntilde\setminus J$, and the extension line recovers the point of the projective bundle. Cohomology and base change make this reconstruction functorial in families; it is unchanged by the scalar line twists inherent in a universal stable family. This also applies to infinitesimal deformations over an Artin local base. Such a deformation, shifted back, is a flat deformation of the ideal sheaf; its determinant stays trivial since $H^1(\OO_X)=H^2(\OO_X)=0$. Cohomology and base change recover the rank-one module of hyperplane sections, and the quotient is a flat Gieseker-stable deformation of the residual root sheaf. Its closed point lies off the point component, so it is represented by the open root locus of the APR moduli space. The extension line then gives the inverse on these deformation functors as well. Thus the stable ambient moduli is smooth along this image, and the reconstruction proves (i).

On the boundary, Lemma~\ref{lem:boundaryelementary} gives
\[
\operatorname{gr}(F_{p,\ell,L}[1])
 \simeq I_{E_L}[1]\oplus\OO_p.
\]
Thus the good-moduli map forgets precisely the line $\ell$ and factors through $\eta$. Corollary~\ref{cor:actual} shows that the stable factors recover both $p$ and $L$, so distinct pairs $(p,L)$ are not S-equivalent. Proposition~\ref{prop:Sigma} computes the incidence fibers: they are reduced points away from $Z$ and $\PP^1$ over $Z$. This proves (ii) and (iii).

Finally, $\Rtilde$ is projective and irreducible. Hence its image is closed and
irreducible. We verify that the smooth rational-quintic locus is dense. On the
dense open locus where $S$ is smooth, write $S=\Bl_{p_1,\ldots,p_4}\PP^2$,
where the four points are in general position. Let $h$ be the pullback of a
line and $e_i$ the exceptional classes, so $H|_S=-K_S=3h-\sum_i e_i$.
The roots are precisely
\[
e_i-e_j\quad(i\ne j),\qquad
\pm(h-e_i-e_j-e_k)\quad(i,j,k\text{ distinct}).
\]
Indeed, writing $D=ah-\sum_i b_i e_i$, the equations $D\cdot H=0$ and
$D^2=-2$ give $\sum_i b_i=3a$ and $\sum_i b_i^2=a^2+2$.
Cauchy's inequality implies $5a^2\le8$, so $a=0,\pm1$; the displayed list
then follows immediately. Up to permuting the points, the corresponding
classes $H-D$ can therefore be written as
\[
\begin{gathered}
(h-e_1)+(2h-e_1-e_3-e_4),\qquad 2h-e_4,\\
(2h-e_1-e_2-e_3)+(2h-e_1-e_2-e_3-e_4).
\end{gathered}
\]
The line pencil $|h-e_1|$ is basepoint-free on $S$, as are the systems of
conics through the indicated subsets of the four points: their only plane
base points are the prescribed points, and blowing up resolves them.
Tensor products of globally generated line bundles are globally generated,
so $|H-D|$ is basepoint-free in every case. Put $C\in|H-D|$. Since
$C^2=3>0$, this linear system is not composed with a pencil. Bertini's
smoothness and irreducibility theorems therefore give a smooth integral
general member. Its degree is $H\cdot C=5$, and adjunction gives
\[
2g(C)-2=C\cdot(C+K_S)=3-5=-2.
\]
Thus the general member is a smooth rational quintic. Openness of smoothness
in this family now gives a nonempty dense open subset of $\Rtilde$ consisting
of smooth rational quintics. It follows that the image of $\Phi$ is exactly
its reduced closure $R$. By (i), $\Phi$ is an isomorphism on a dense open set,
hence birational; projectivity of the source makes it proper.
\end{proof}

\subsection{The formal contraction of the principal branch}
We compute the normal bundle and formal neighborhoods of the exceptional flag locus, determining $R^\nu$ before the ambient deformation calculation.

Let
\[
a:Z=\Gr_X(2,K)\longrightarrow X,
\qquad
q_F:F=\PP_Z(\cS)\longrightarrow Z,
\]
where $\cS\subset a^*K$ is the tautological rank-two bundle. Write
\[
\lambda:=\OO_F(-1)\subset q_F^*\cS,
\qquad
\cQ_F:=q_F^*\cS/\lambda,
\]
and set $B:=\widetilde\Sigma=\Rtilde|_J$.

\begin{lemma}\label{lem:formalnormal}
There is a canonical identification
\[
N_{F/B}\simeq \cQ_F^\vee\otimes(aq_F)^*(W/K)
 \simeq \OO_F(-1)\otimes q_F^*G_0,
\]
where
\[
G_0:=(\det\cS)^\vee\otimes a^*(W/K),
\qquad \rk G_0=3.
\]
Moreover, if $L_H:=a^*\OO_X(H)$, then
\[
N_{B/\Rtilde}|_F\simeq \OO_F(-1)\otimes q_F^*L_H.
\]
Consequently, Zariski locally on $Z$ the normal sequence splits and
\begin{equation}\label{eq:globalformalnormal}
N_{F/\Rtilde}\simeq \OO_F(-1)\otimes q_F^*G,
\qquad G:=G_0\oplus L_H,
\end{equation}
with $G_0\subset G$ the displayed direct summand.
\end{lemma}
\begin{proof}
A point of $F$ is a triple $(p,\ell,L)$ with $\ell\subset L\subset K_p$. Keeping $(p,\ell)$ fixed, the incidence space $B$ moves the quotient line $L/\ell$ inside $W_p/\ell$, whereas $F$ is the locus where that quotient line stays inside $K_p/\ell$. Hence
\[
N_{F/B}|_{(p,\ell,L)}
 \simeq \Hom(L/\ell,W_p/K_p),
\]
which globalizes to the first displayed formula. Since
\[
\det(q_F^*\cS)=\lambda\otimes\cQ_F,
\]
one has
\[
\cQ_F^\vee\simeq
\lambda\otimes q_F^*(\det\cS)^\vee,
\]
giving the formula for $G_0$.

Because $B=\Rtilde\times_{\Ntilde}J$, its normal line in $\Rtilde$ is the pullback of $N_{J/\Ntilde}$. Equation~\eqref{eq:exceptionalnormal} gives
\[
N_{J/\Ntilde}\simeq \OO_J(-1)\otimes q_J^*\OO_X(H).
\]
On $F\to J$ the tautological line $\OO_J(-1)$ pulls back to $\lambda$, so the second displayed formula follows. Thus
\[
0\to \OO_F(-1)\otimes q_F^*G_0
 \to N_{F/\Rtilde}
 \to \OO_F(-1)\otimes q_F^*L_H\to0.
\]
After shrinking to an affine open $U\subset Z$, its extension class lies in
\[
H^1\!\left(F_U,q_F^*(G_0\otimes L_H^{-1})\right)=0,
\]
because $R^1q_{F*}\OO_F=0$. Hence the sequence splits Zariski locally on $Z$, proving \eqref{eq:globalformalnormal}.
\end{proof}

\begin{theorem}\label{thm:formallinearization}
Zariski locally on $Z$, the formal pair $(\Rtilde,B)$ along $F$ is isomorphic to the formal pair
\begin{equation}\label{eq:formalpair}
\left(
\Tot_F(\OO_F(-1)\otimes q_F^*G),
\Tot_F(\OO_F(-1)\otimes q_F^*G_0)
\right)
\end{equation}
along the common zero section, with $G_0\subset G$ as in Lemma~\ref{lem:formalnormal}.
\end{theorem}
\begin{proof}
Work over an affine open $U\subset Z$ on which the normal sequence of Lemma~\ref{lem:formalnormal} is split. Put
\[
N:=\OO_F(-1)\otimes q_F^*G,
\qquad
N_0:=\OO_F(-1)\otimes q_F^*G_0.
\]

After shrinking $U$ further, trivialize $\cS$, $G_0$, and $L:=G/G_0$ on $U$, so $F_U\simeq\PP^1\times U$ and $N=N_0\oplus N_1$, where $N_1=\OO_F(-1)\otimes q_F^*L$. The formal-neighborhood induction uses coordinate changes homogeneous in the normal variables. Horizontal terms of degree $m$ form
\[
T_{F_U}\otimes\Sym^mN^\vee.
\]
Vertical terms preserving the subbundle $N_0$ form the actual kernel sheaf
\[
\mathcal V_m=
\ker\!\left(N\otimes\Sym^mN^\vee
\longrightarrow N_1\otimes\Sym^mN_0^\vee\right).
\]
With the chosen splitting it is the direct sum of
\[
N_0\otimes\Sym^mN^\vee
\quad\text{and}\quad
N_1\otimes
\bigoplus_{j=1}^m\Sym^{m-j}N_0^\vee\otimes\Sym^jN_1^\vee.
\]
The second summand expresses that the component normal to the divisor is divisible by its defining normal coordinate. Each vertical summand has fiber degree $m-1$; the horizontal summands have fiber degrees $m$ and $m+2$, using $T_{\PP^1}=\OO_{\PP^1}(2)$.

For every $m\ge1$ these sheaves have vanishing $H^1$ on $F_U$, since $U$ is affine and $H^1(\PP^1,\OO(k))=0$ for $k\ge-1$. The first-order splitting and the differences between successive lifts are governed by these horizontal and pair-preserving vertical sheaves. Their difference cocycles can therefore be killed at each order. This verifies the required vanishing for the kernel sheaves themselves; no vanishing is inferred merely from their being subquotients of other sheaves.

Thus the formal coordinates can be straightened successively to every order, compatibly with $B$. Passing to the inverse limit gives the isomorphism of completed pairs in \eqref{eq:formalpair}.
\end{proof}

\begin{theorem}\label{thm:normalization}
Let
\[
\Rtilde\xrightarrow{\psi}R^\nu\xrightarrow{\nu}R
\]
be the Stein factorization of $\Phi$. Then $R^\nu$ is the normalization of $R$, and $\psi$ is a small resolution. Its exceptional locus is $F$, of dimension six, and its only nontrivial fibers are the $\PP^1$-fibers of $F\to Z$.

At a point $z\in Z$, after local trivialization,
\begin{equation}\label{eq:formalring}
\widehat{\OO}_{R^\nu,z}\simeq
\widehat{\OO}_{Z,z}[[x_{11},x_{12},x_{13},x_{14},x_{21},x_{22},x_{23},x_{24}]]/I_2(X),
\end{equation}
where $X=(x_{ij})$ is a $2\times4$ matrix. The contraction of the boundary has the analogous $2\times3$ determinantal formal model. These are formal local statements, not a global identification of $R^\nu$ with an affine bundle-cone.
\end{theorem}
\begin{proof}
The source $\Rtilde$ is smooth and $R$ is integral. A proper birational morphism from a normal variety factors through the normalization; the induced proper birational morphism to the normal target has direct image of the structure sheaf equal to the structure sheaf. Hence the finite morphism in the Stein factorization is precisely the normalization $\nu:R^\nu\to R$.

Theorem~\ref{thm:principalcomponent} gives connected geometric fibers for $\Phi$: the only nontrivial ones are the $\PP^1$ fibers over $Z$. By the defining property of Stein factorization, the fibers of $\psi$ are the connected components of the fibers of $\Phi$, so the same $\PP^1$'s are exactly the positive-dimensional fibers of $\psi$. Since $\dim F=6$ and $\dim\Rtilde=10$, no divisor is contracted. Away from $F$ the map is proper, quasi-finite and birational onto a normal scheme, hence an isomorphism. Thus $\psi$ is a small resolution.

By Theorem~\ref{thm:formallinearization} and the theorem on formal functions, the completed contraction algebra along $Z$ is the completion of
\begin{equation}\label{eq:conering}
\bigoplus_{m\ge0}\Sym^m\cS^\vee\otimes\Sym^mG^\vee.
\end{equation}
Indeed
\[
q_{F*}\OO_F(m)=\Sym^m\cS^\vee,
\qquad R^iq_{F*}\OO_F(m)=0\quad(i>0,m\ge0).
\]
After trivializing the rank-two bundle $\cS$ and the rank-four bundle $G$, \eqref{eq:conering} is the Segre ring generated by the products $u_iv_j$; its defining ideal is the ideal of $2\times2$ minors of a generic $2\times4$ matrix. Completion along the vertex gives \eqref{eq:formalring}. Replacing $G$ by the rank-three subbundle $G_0$ gives the boundary $2\times3$ model.
\end{proof}

\begin{corollary}\label{cor:singularities}
The normalization $R^\nu$ is Cohen--Macaulay, its reduced singular locus is $Z$ of codimension five, and its transverse singularities are rational rank-one determinantal singularities. In characteristic zero they are rational singularities. None of these statements by itself proves $R=R^\nu$.
\end{corollary}
\begin{proof}
The affine cone over $\PP^1\times\PP^3$ is a normal Cohen--Macaulay determinantal variety, smooth away from its vertex. Its standard resolution by $\Tot_{\PP^1}(\OO(-1)^{\oplus4})$ has vanishing higher direct images of the structure sheaf because $H^i(\PP^1,\OO(m))=0$ for $i>0$ and $m\ge0$. Thus the transverse singularity is rational. Theorem~\ref{thm:normalization} identifies the completed local rings along $Z$ with this model over the smooth base $Z$, while $\psi$ is an isomorphism away from $F$; hence the asserted Cohen--Macaulay, singular-locus, codimension and rational-singularity statements follow.
\end{proof}

\subsection{Classification of geometric points}
The local deformation calculation below needs the global fact that there are no further geometric branches. We first establish that statement independently of scheme structure.

\begin{lemma}\label{lem:pointsexhaust}
Every geometric point of $M^{ss}$ belongs to $R\cup C$, where
\[
C:=X\times\Gr(2,V),\qquad (p,L)\longmapsto[I_{E_L}[1]\oplus\OO_p].
\]
\end{lemma}
\begin{proof}
Let $E$ be a polystable representative. Lemma~\ref{lem:phaseone} and Corollary~\ref{cor:budgetouter} give
\[
0\longrightarrow F[1]\longrightarrow E\longrightarrow Q\longrightarrow0,
\qquad \ch(F)=v_5+m\pt,\quad m=\operatorname{length}Q\le1.
\]
If $m=1$, then $Q=\OO_p$ and Proposition~\ref{prop:d5reconstruction} gives $F\simeq I_{E_L}$ for a unique $L$. Hence the S-equivalence class of $E$ is the point $(p,L)$ of $C$.

Assume $m=0$. As in the first paragraph of the proof of Proposition~\ref{prop:d5reconstruction}, vertical transport to large $\alpha$ shows that $F$ is a rank-one torsion-free sheaf. Since $c_1(F)=0$ and $\operatorname{Pic}(X)=\ZZ H$, its reflexive hull is $F^{\vee\vee}\simeq\OO_X$, so $F=I_Z$ for a subscheme of codimension at least two. Let
\begin{equation}\label{eq:saturation}
0\longrightarrow F\longrightarrow I_C\longrightarrow Q_0\longrightarrow0
\end{equation}
be the saturation with respect to the one-dimensional part. Then $C$ is pure one-dimensional, hence Cohen--Macaulay on the smooth threefold, and $Q_0$ is zero-dimensional of length $r$. The ideal $I_C$ is slope-stable, and Proposition~\ref{prop:nolarger} transports it to the outer point as a tilt-stable object of class $v_5+r\pt$. Corollary~\ref{cor:budgetouter} therefore gives $r\le1$.

If $r=1$, Proposition~\ref{prop:d5reconstruction} gives $I_C\simeq I_{E_L}$, and shifting \eqref{eq:saturation} shows that $F[1]$ is S-equivalent to $I_{E_L}[1]\oplus\OO_p$. Thus the point again lies in $C$.

It remains to treat $r=0$, so $F=I_C$ for a Cohen--Macaulay quintic with Hilbert polynomial $5m+1$. At the terminal point,
\[
\chi(\OO_X(-1),I_C)=\chi(I_C(1))=1.
\]
Moreover
\[
\Ext^2(\OO_X(-1),I_C)
 \simeq\Hom(I_C,\OO_X(-3)[1])^\vee=0
\]
by the terminal tilt slopes, and
\[
\Ext^3(\OO_X(-1),I_C)
 \simeq\Hom(I_C,\OO_X(-3))^\vee=0.
\]
Hence $\Hom(\OO_X(-1),I_C)\ne0$. A nonzero map is injective and determines a hyperplane section $S=(s=0)$ containing $C$:
\begin{equation}\label{eq:residualfromsaturated}
0\longrightarrow\OO_X(-1)\xrightarrow{s} I_C
\longrightarrow i_*I_{C/S}\longrightarrow0.
\end{equation}
After twisting by $\OO_X(1)$, the quotient has the APR class $w$. It is rank-one torsion-free on the integral surface $S$, hence Gieseker-stable. By \cite[Proposition~4.1]{APR}, the Gieseker moduli of class $w$ is the union of the root component and the point-on-hyperplane component, with their intersection contained in the latter. Suppose first that the quotient in \eqref{eq:residualfromsaturated} lies in the point component. Before twisting it is then $i_*I_{p/S}(-1)$. Writing an extension class as $g$ modulo $s$, the same pushout calculation as in Lemma~\ref{lem:boundaryelementary} gives
\[
0\longrightarrow I_C\longrightarrow I_{E_{\langle s,g\rangle}}\longrightarrow\OO_p\longrightarrow0.
\]
Thus $I_C$ has a nontrivial zero-dimensional saturation defect, contradicting $r=0$. Hence the residual point lies in the root component but outside its intersection with the point component. There the full five-dimensional extension space is the principal root-extension space, so \eqref{eq:residualfromsaturated} gives a point of $R$. This completes the proof.
\end{proof}

\subsection{Cross deformations and the complete equivariant local model}
Fix a corrected closed point
\[
c=(p,L)\in C,\qquad E_c=I_{E_L}[1]\oplus\OO_p,
\]
and put $S_c=\widehat{\OO}_{C,c}$. Choose a local frame $L=\langle s_1,s_2\rangle$ and a local trivialization of $\OO_X(H)$ along the moving point. The first jets of the two sections give a $4\times2$ matrix
\begin{equation}\label{eq:firstjetmatrix}
\begin{pmatrix}
e_1&e_2\\
m_{11}&m_{21}\\
m_{12}&m_{22}\\
m_{13}&m_{23}
\end{pmatrix}.
\end{equation}
Here $e_i=s_i(p)$; on the locus $e_1=e_2=0$ the lower $3\times2$ block represents the map $L\to(W/K)_p$. Near $Z$, the five parameters along $Z$, the two $e_i$, and the six $m_{ij}$ form a regular system of parameters, exactly as in the Grassmannian chart of Proposition~\ref{prop:Sigma}.

\begin{lemma}\label{lem:crossext}
At $c=(p,L)$, the cross-Ext spaces of $E_c=I_{E_L}[1]\oplus\OO_p$ have dimensions
\[
\begin{array}{c|cc}
 &p\in E_L&p\notin E_L\\ \hline
\Ext^1(I_{E_L}[1],\OO_p)&2&1\\
\Ext^1(\OO_p,I_{E_L}[1])&1&0.
\end{array}
\]
The total tangent dimension is respectively $16$ or $14$. The two cross directions have opposite weights for the effective stabilizer $T\simeq\mathbb G_m$.
\end{lemma}
\begin{proof}
The Koszul resolution
\[
0\to\OO_X(-2)\xrightarrow{(-s_2,s_1)}\OO_X(-1)^{\oplus2}
\to I_{E_L}\to0
\]
computes $\Hom(I_{E_L},\OO_p)$ as the kernel of
$\CC^2\xrightarrow{(-e_2,e_1)}\CC$. Applying local duality computes
$\Ext^2(\OO_p,I_{E_L})$ as the kernel of
$\CC\xrightarrow{(e_1,e_2)^t}\CC^2$.
These give the table. The linear-section ideal has ten tangent directions and is unobstructed: its normal bundle is $\OO_{E_L}(H)^{\oplus2}$, whose $H^1$ vanishes by the Koszul cohomology calculation. The point has its three-dimensional smooth deformation space. Adding the diagonal and cross terms gives the dimensions. The two stable nonisomorphic summands have automorphism group $\mathbb G_m^2$; the diagonal scalar acts trivially and the two off-diagonal Hom spaces have opposite characters.
\end{proof}

The relative cross-Hom complexes give the linear relations
\[
a_1e_2-a_2e_1=0,\qquad be_1=be_2=0.
\]
The mixed differentials give the displayed linear relations; first derivatives suggest $b(a_1m_{2k}-a_2m_{1k})$. The next theorem proves their completeness. Its first step obtains a relative Stasheff syzygy and computes the mixed Yoneda term. Its second step uses that syzygy to normalize the full equations, then absorbs the higher invariant terms by formal coordinates on each rank stratum.

\begin{theorem}\label{thm:kuranishi}
For every corrected closed point $c\in C$, there is a $\mathbb G_m^2$-equivariant semi-universal slice whose completed deformation ring is, compatibly with the diagonal family over $C$,
\begin{equation}\label{eq:slicering}
A_c=S_c[[a_1,a_2,b]]\Big/
\left(
a_1e_2-a_2e_1,\ be_1,\ be_2,\
b(a_1m_{2k}-a_2m_{1k})\ (1\le k\le3)
\right).
\end{equation}
The diagonal scalar acts trivially. The effective torus has weights
\[
\operatorname{wt}(a_1)=\operatorname{wt}(a_2)=1,
\qquad\operatorname{wt}(b)=-1,
\qquad\operatorname{wt}(S_c)=0,
\]
and the family $C$ corresponds to $a_1=a_2=b=0$. The identification is noncanonical, restricts to the given diagonal family, and is valid on every commutative local Artin $\CC$-algebra with residue field $\CC$.
\end{theorem}
\begin{proof}
Put
\[
F=I_{E_L}[1],\qquad P=\OO_p,\qquad E=F\oplus P,\qquad S=S_c.
\]
Use the full strong exceptional collection
$\langle\cU,\cQ^\vee,\OO_X,\OO_X(1)\rangle$
of \cite{OrlovV5}, recalled in \cite[Section~3]{APR}. Its direct sum is a tilting bundle, so the derived equivalence represents $E$ by a bounded h-projective complex of finitely generated projectives over its endomorphism algebra. Since $F,P$ lie in a common heart, $\Ext^i(E,E)=0$ for $i<0$. Thus \cite[Theorem~11.6(a)]{ELO} identifies derived and homotopy deformations on commutative local Artin $\CC$-algebras with residue field $\CC$. By \cite[Proposition~6.1 and Theorem~8.1]{ELO} and minimal-model transfer \cite[Sections~3 and~7]{Keller}, these are governed by the full Maurer--Cartan groupoid of a strictly unital minimal $A_\infty$ category on $F,P$ over $\CC\oplus\CC$.
In particular these equations give a complete deformation hull on all the
Artin bases used here. The Ext-quiver formulation is explained in
\cite[Sections~4--5]{Toda}. The idempotent decomposition gives the torus
weights and the allowed paths of all operations. Since
$\Ext^0(E,E)=\CC\oplus\CC$, strict unitality makes every higher operation
with a degree-zero input vanish. Thus degree-zero gauge is exactly scalar
change of basis at the two vertices.

Suppose first that $p\in E_L$. The Koszul resolution and local duality, together with Lemma~\ref{lem:crossext}, give
\begin{equation}\label{eq:formalexttable}
\begin{array}{c|rrrr}
 &\Ext^1(F,F)&\Ext^1(P,P)&\Ext^1(F,P)&\Ext^1(P,F)\\ \hline
\dim&10&3&2&1
\end{array}
\end{equation}
and
\begin{equation}\label{eq:formalobstructions}
\Ext^2(E,E)=\Ext^2(F,P)\oplus\Ext^2(P,F)\oplus\Ext^2(P,P),
\qquad \dim=(1,2,3),
\end{equation}
whereas
\begin{equation}\label{eq:formalextthree}
\Ext^3(E,E)=\Ext^3(P,P)\simeq\CC.
\end{equation}
Here $\Ext^2(F,F)=0$: the complete-intersection normal bundle is $\OO_{E_L}(H)^{\oplus2}$ and its higher cohomology vanishes. The ideal-sheaf deformation space is the smooth Grassmannian. Also $\Ext^3(F,F)=0$ by Serre duality, since $\Hom(I_{E_L},I_{E_L}(-2H))=0$. The other degree-three cross groups vanish by the length-one locally free resolution of the ideal and dimension three. Negative Ext groups vanish since the two objects belong to a common heart.

Both diagonal deformation functors are smooth and their product is $\Spf S$. Accordingly take the ambient formal power-series ring
\[
\mathscr R=S[[a_1,a_2,b]],\qquad q_i=a_i b.
\]
A Kuranishi map has six components, denoted
\begin{equation}\label{eq:formalkappa}
f,\quad r_1,r_2,\quad k_1,k_2,k_3,
\end{equation}
with weights $1,-1,-1,0,0,0$. All vanish along $a=b=0$. In particular
\[
r_i=bG_i(S,q),\qquad k_j\in(q_1,q_2)\subset\mathscr R^T=S[[q_1,q_2]].
\]
The point corner of the minimal model can be chosen to be the exterior algebra on $T_pX$, with no higher operations. Indeed, over $\CC[[x_1,x_2,x_3]]$, the contraction endomorphisms of the Koszul resolution anticommute, are closed, and generate its endomorphism cohomology. They give a dg-algebra quasi-isomorphism from the exterior algebra to that endomorphism algebra. An $A_\infty$ isomorphism putting the minimal point corner in this form extends to the two-object category: extend its bar-coalgebra automorphism by taking its higher Taylor coefficients to be zero on all other paths, and transport the full structure. This imposes no formality condition on the mixed operations. We use the suspension sign convention in which the Maurer--Cartan curvature is $\kappa(\xi)=\sum_n m_n(\xi,\ldots,\xi)$.

\medskip\noindent\emph{Step 1: paired normalization along the diagonal.}
The coordinates and obstruction frames may be chosen so that the terms involving just one cross variable are
\begin{equation}\label{eq:formalonecross}
\begin{aligned}
f&=a_1e_2-a_2e_1+\text{terms with at least three cross variables},\\
r_i&=be_i+\text{terms with at least three cross variables}.
\end{aligned}
\end{equation}
With these choices, the Stasheff curvature identity in the one-dimensional space \eqref{eq:formalextthree} has the form
\begin{equation}\label{eq:formalbianchi}
bU f-v_1r_2+v_2r_1+\sum_{j=1}^3 w_jk_j=0,
\end{equation}
where
\[
U\equiv1\pmod{(q)},\qquad
\binom{v_1}{v_2}=V_0(S,q)\binom{a_1}{a_2},\quad
V_0\equiv I_2\pmod{(q)},\qquad w_j\in(q).
\]
Along the incidence locus $e_1=e_2=0$, the part of the point obstruction with two cross variables is
\begin{equation}\label{eq:formalyoneda}
k_j=q_1m_{2j}-q_2m_{1j}\pmod{(q)^2,(e_1,e_2)(q)}.
\end{equation}
This congruence holds over the complete diagonal base, not only at its closed point.

\smallskip\noindent\emph{Relative mixed complexes.}
Let $\eta$ be the universal diagonal Maurer--Cartan element over $S$, and twist the minimal category by $\eta$. Its mixed complexes are free complexes of ranks $(2,1)$ and $(1,2)$ in degrees $(1,2)$. Twisting computes derived Hom between the corresponding universal families; this follows by applying the minimal-model quasi-equivalence to twisted objects, or by induction over the Artin quotients of $S$. The universal resolution
\[
0\longrightarrow\OO_X(-2H)\xrightarrow{(-s_2,s_1)}
\OO_X(-H)^{\oplus2}\longrightarrow I_{E_L}\longrightarrow0
\]
identifies the first with $[S^2\xrightarrow{(e_2,-e_1)}S]$, after a sign choice. This identification is a change of free bases: both descriptions are bounded minimal free complexes at the closed point. A derived isomorphism between such complexes is represented by a chain homotopy equivalence; its reduction modulo the maximal ideal is a degreewise isomorphism, hence so is the map itself by Nakayama's lemma.

\smallskip\noindent\emph{Paired differentials.}
All relative derived Hom complexes in this paragraph are taken on
$X_S=X\times\Spf S$, or equivalently on its compatible Artin thickenings.
Write $\omega=\operatorname{pr}_X^*K_X$, and let
$D_S(F_S,P_S)=\mathbf R\operatorname{Hom}_{X_S}(F_S,P_S)$.
Relative Serre duality and tensor adjunction give
\[
\begin{aligned}
D_S(F_S,P_S)^\vee
&\simeq \mathbf R\operatorname{Hom}_{X_S}(P_S,F_S\otimes\omega)[3]\\
&\simeq \mathbf R\operatorname{Hom}_{X_S}(P_S\otimes\omega^{-1},F_S)[3],
\end{aligned}
\]
where $(-)^\vee=\mathbf R\operatorname{Hom}_S(-,S)$.
Choose a trivialization of $K_X$ on the formal neighbourhood of the moving
point. Its restriction identifies $P_S\otimes\omega^{-1}$ with $P_S$,
and hence identifies the last complex with $D_S(P_S,F_S)[3]$.
The same trivialization identifies the relative Serre trace with a map
\[
\operatorname{tr}_{P_S}:\Ext^3_{X_S}(P_S,P_S)\xrightarrow{\ \sim\ }S.
\]
Thus ordinary composition into the point corner, followed by this trace,
realizes the above duality. Only the canonical bundle along the point
family is trivialized; no identification $F_S\otimes\omega\simeq F_S$
is used.

In the twisted minimal model the point-corner differential is zero.
Projection of its degree-three term to $S$ by $\operatorname{tr}_{P_S}$
is therefore a chain map. The $n=2$ twisted Stasheff identity makes
$\operatorname{tr}_{P_S}\circ m_2^\eta$ a pairing of the two mixed
complexes: for homogeneous $u,v$ its differential identity is
\[
\operatorname{tr}_{P_S}\bigl(m_2^\eta(m_1^\eta u,v)
       +(-1)^{|u|}m_2^\eta(u,m_1^\eta v)\bigr)=0.
\]
Here the order of the composable factors is chosen so that the output
lies in the point corner; suspension changes only the corresponding
sign convention. At the closed point these pairings are the perfect
Serre pairings. Their matrices in complementary degrees $(1,2)$ and
$(2,1)$ consequently have unit determinants over $S$. This supplies
perfect pairings of the free terms over the whole diagonal base, not
merely pairings of their closed-point cohomology.

Fix the Koszul bases in the first complex, the dual basis in reverse degree one, and the skew-dual bases in reverse degree two. In these bases the two pairings are respectively $fb$ and $a_1r_2-a_2r_1$. The $n=2$ twisted Stasheff identity says that composition is a chain map. Since the point-corner differential is zero, it forces the reverse differential to be $(e_1,e_2)^t$. Thus the one-cross curvature is exactly \eqref{eq:formalonecross}, and the cross-linear coefficients of the trace identity are $b,-a_1,a_2$, over all of $S$.

\smallskip\noindent\emph{The full curvature identity.}
The curved Stasheff identity is
\[
\sum_{r,t\ge0}\pm m_{r+1+t}
\bigl(\xi^{\otimes r},\kappa(\xi),\xi^{\otimes t}\bigr)=0,
\qquad \kappa(\xi)=\sum_{n\ge2}m_n(\xi^{\otimes n}).
\]
It is the identity $(m^\xi)_1((m^\xi)_0)=0$, obtained by conjugating the bar coderivation, whose square is zero; it holds before imposing any Maurer--Cartan equation. Project it to $\Ext^3(P,P)$. It is linear in the six components of $\kappa$. Weight considerations make its coefficient of $f$ divisible by $b$, and its coefficients of $r_i$ linear combinations of the $a_j$ over $S[[q]]$. Their values along $q=0$ are the perfect pairings just normalized. The coefficients of $k_j$ vanish along $q=0$: there only the point corner can contribute, and in its exterior algebra the degree-one element commutes with a degree-two element; its higher operations vanish. Thus the remaining weight-zero coefficients lie in $(q)$, giving \eqref{eq:formalbianchi}.

\smallskip\noindent\emph{The relative Yoneda coefficient.}
Work over $S_0=S/(e_1,e_2)$. The differentials in both mixed complexes vanish, so their bases represent the actual relative cross-Ext classes. In coordinates $x_k$ centered on the moving point write
\[
s_i=e_i+\sum_k m_{ik}x_k\pmod{(x)^2}.
\]
Put $R_0=S_0[[x_1,x_2,x_3]]$ and choose $t_{ik}\in R_0$ with $s_i=\sum_k x_kt_{ik}$, so $t_{ik}|_{x=0}=m_{ik}$. Let $Q_\bullet$ be the point Koszul resolution, with
\[
d_Q(Q_3)=x_1Q_{23}-x_2Q_{13}+x_3Q_{12}.
\]
In $\Hom_{R_0}(Q_\bullet,R_0^2)$ take the degree-two cochain $h$ whose values on $Q_{23},Q_{13},Q_{12}$ are respectively
\[
(-t_{21},t_{11}),\qquad (t_{22},-t_{12}),\qquad (-t_{23},t_{13}).
\]
Then $h\,d_Q=(-s_2,s_1)$. To verify that its image is a generator,
write $I$ for the restriction of $I_{E_L}$ to this formal neighbourhood
and consider its Koszul presentation
\[
0\longrightarrow R_0\xrightarrow{\iota=(-s_2,s_1)}R_0^2
\longrightarrow I\longrightarrow0.
\]
The variables $x_1,x_2,x_3$ form an $R_0$-regular sequence, so the point
Koszul resolution gives
\[
\Ext^j_{R_0}(S_0,R_0)=0\quad(j\ne3),\qquad
\Ext^3_{R_0}(S_0,R_0)\simeq S_0.
\]
Since $s_1,s_2\in(x_1,x_2,x_3)$, the map induced by $\iota$ on
degree-three Ext is zero. The connecting map is therefore an isomorphism
\[
\partial:\Ext^2_{R_0}(S_0,I)
\xrightarrow{\ \sim\ }\Ext^3_{R_0}(S_0,R_0)\simeq S_0.
\]
Let $\bar h$ denote the image of $h$ in $\operatorname{Hom}_{R_0}(Q_2,I)$.
It is a cocycle, and the lift $h$ computes its connecting class:
$h\,d_Q=\iota\circ 1$ on $Q_3=R_0$. Thus
$\partial[\bar h]=\pm1$, with the sign fixed by the Koszul convention.
In particular $[\bar h]$ is a generator over $S_0$, rather than a
possibly nonunit multiple of one.
The same presentation gives $\Ext^1_{R_0}(I,S_0)\simeq S_0$.
In the compatible Koszul orientation and Serre trace, $[\bar h]$ is
dual to its unit generator, up to the same fixed sign. We use this
normalization for the reverse cross variable $b$.
Composing $\bar h$ with the map $I\to S_0$ given on the two generators
by $(a_1,a_2)$ produces the coefficients of $a_1ds_2-a_2ds_1$
under the volume identification $\Lambda^2T\simeq T^*$, with this
fixed sign choice. Multiplying by $b$ gives \eqref{eq:formalyoneda}. This calculation is over $S_0$, so it proves the relative assertion, including singular and nonreduced linear sections. Terms involving additional alternating cross arrows belong to $(q)^2$.

\medskip\noindent\emph{Step 2: elimination of the higher mixed terms.}
Set $b'=bU$ and $(a'_1,a'_2)^t=V_0(a_1,a_2)^t$ in \eqref{eq:formalbianchi}. These are equivariant formal coordinate changes and are the identity to first order in the cross variables over $S$. They preserve the diagonal base. Rewriting the weight-zero coefficients gives
\[
w_j=b'\bigl(a'_1c_{1j}+a'_2c_{2j}\bigr).
\]
Replace obstruction generators by
\[
r'_2=r_2-b'\sum_jc_{1j}k_j,\qquad
r'_1=r_1+b'\sum_jc_{2j}k_j.
\]
This is an invertible triangular change of the six generators and hence preserves their ideal. The identity becomes
\[
b'f-a'_1r'_2+a'_2r'_1=0.
\]
Every weight-$-1$ series is divisible by $b'$, so write $r'_i=b'G_i$. Dividing this identity in the ambient power-series domain by $b'$ gives
\begin{equation}\label{eq:formalexactskew}
f=a'_1G_2-a'_2G_1.
\end{equation}
Drop the primes on the cross coordinates. We have obtained an exact presentation of the Kuranishi ideal by
\begin{equation}\label{eq:formalnormalizedkappa}
a_1G_2-a_2G_1,\quad bG_1,\ bG_2,\quad
q_1H_{2j}-q_2H_{1j}\quad(1\le j\le3),
\end{equation}
where $G_i(S,0)=e_i$ and initially $H_{ij}(S,0)\equiv m_{ij}\bmod(e_1,e_2)$. To make the latter equality exact, subtract from each $k_j$ appropriate multiples of $a_ir'_\ell$. Such a subtraction changes its coefficient of $q_i$ along $q=0$ by an arbitrary multiple of $e_\ell$, and does not change the ideal. Thus we can and do arrange
\begin{equation}\label{eq:formalidentitybase}
G_i(S,0)=e_i,\qquad H_{ij}(S,0)=m_{ij}.
\end{equation}
No assumption has been made about the remaining $q$ dependence of $G,H$.

At a point $c\in Z$, first-jet evaluation $V\to J^1\OO_X(H)|_p$ is surjective. At $L\subset K_p$, the tangent map from the Grassmannian chart to $\Hom(L,J^1\OO_X(H)|_p)$ is surjective. Consequently $e_1,e_2$ and the six $m_{ij}$ are independent parameters on $C$, complemented by five parameters $z_1,\ldots,z_5$ along $Z$. In particular
\[
S\simeq\CC[[z_1,\ldots,z_5,e_1,e_2,m_{11},\ldots,m_{23}]].
\]
In \eqref{eq:formalnormalizedkappa} make the weight-zero coordinate change
\[
e'_i=G_i(S,q),\qquad m'_{ij}=H_{ij}(S,q),\qquad z'_\ell=z_\ell.
\]
By \eqref{eq:formalidentitybase} these series reduce to $e_i,m_{ij}$ when $q=0$. Since $q_i=a_ib$ has cross degree two, the full ambient change has invertible linear part and hence a formal inverse. It restricts to the identity on $a=b=0$, so it fixes the embedded diagonal family, including its scheme structure. In these coordinates \eqref{eq:formalnormalizedkappa} is precisely the ideal in \eqref{eq:slicering}. This absorbs all higher terms simultaneously at the deepest stratum.

Suppose next that $e_1=e_2=0$ and $\rk M(c)=1$. An invertible entry permits Gaussian row and column operations on $M$ over $S$, putting it in the form
\[
M=\begin{pmatrix}1&0&0\\0&u&v\end{pmatrix}.
\]
Row operations act simultaneously on the frame of $L$, on $e$, and on the positive cross vector $a$; the skew relations change by a unit. Column operations change the point-obstruction frame. Thus they preserve the form of the asserted presentation. The incidence scheme $\Sigma$ is smooth at this point, and $e_1,e_2,u,v$ are independent parameters on $C$; complete these four parameters by nine coordinates along $\Sigma$.

Apply Gaussian operations to $H(S,q)$ in \eqref{eq:formalnormalizedkappa}. Because of \eqref{eq:formalidentitybase}, they may be chosen to reduce to the identity at $q=0$, after the preceding normalization of $M$. The row operation is accompanied by the same equivariant change on $a$ and $G$; the column operation is a change of obstruction generators. The point equations become
\[
-q_2,\qquad q_1U,\qquad q_1V,
\]
where $U(S,0)=u$ and $V(S,0)=v$. The four weight-zero series $G_1,G_2,U,V$, together with the nine unchanged diagonal coordinates, give a formal coordinate system. This change again restricts to the identity on $C$, and it gives \eqref{eq:slicering} in the rank-one matrix chart.

If $e_1=e_2=0$ and $\rk M(c)=2$, the same operations reduce the point equations to $q_1,q_2$. Modulo these equations, \eqref{eq:formalidentitybase} makes the other relations exactly
\[
a_1e_2-a_2e_1,\qquad be_1,\qquad be_2.
\]
This is also the reduction of \eqref{eq:slicering}, since a $2\times2$ minor of $M$ is a unit.

Finally suppose $p\notin E_L$. The cross tangent dimensions are $(1,0)$ by Lemma~\ref{lem:crossext}. Thus all non-diagonal tangent variables have the same nonzero weight. There is no obstruction space of that weight, and the diagonal family is smooth. The equivariant Kuranishi map is therefore zero, and the hull is $S[[a]]$. In \eqref{eq:slicering} an $e_i$ is a unit; eliminating $b$ and one $a_i$ gives exactly $S[[a]]$.

These cases cover $C$. Every manipulation was either an invertible change of obstruction generators or an equivariant automorphism of a complete ambient power-series ring. Therefore it identifies the complete Maurer--Cartan formal schemes, with their degree-zero gauge action, on every commutative local Artin $\CC$-algebra with residue field $\CC$. This proves the theorem.
\end{proof}

\medskip\noindent\emph{Compatibility remark.}
The formal coordinate change is allowed to mix a diagonal coordinate with the invariants $q_i$, but is the identity after setting $a=b=0$. Thus the result preserves the diagonal family; it does not prescribe a canonical projection of all nearby deformations onto that family. This is exactly the compatibility needed for the completed map from the corrected component.

\begin{proposition}\label{prop:kuranishi}
The invariant ring of the algebra in \eqref{eq:slicering} is
\begin{equation}\label{eq:ambientcompletedring}
A_c^T\simeq
S_c[[q_1,q_2]]\Big/
\left(
e_iq_j\ (1\le i,j\le2),\
q_1m_{2k}-q_2m_{1k}\ (1\le k\le3)
\right),\qquad q_i=a_ib.
\end{equation}
This is $\widehat{\OO}_{M^{ss},c}$, and the completed map from $C$ is the quotient by $(q_1,q_2)$.
\end{proposition}
\begin{proof}
Every invariant monomial in the formal power-series algebra is generated over $S_c$ by $a_1b,a_2b$. Since the torus is linearly reductive, invariants are exact at each finite quotient, and passing to the inverse limit computes the completed invariant ring. Multiplying the relation $a_1e_2-a_2e_1$ by a weight-$-1$ monomial gives a multiple of $q_1e_2-q_2e_1$. Multiplying $be_i$ by a weight-$1$ monomial gives the ideal generated by $e_iq_j$. These already imply the former relation. The weight-zero relations become $q_1m_{2k}-q_2m_{1k}$. Every homogeneous multiple is accounted for in this way, proving \eqref{eq:ambientcompletedring}.

Apply the local quotient comparison recorded in Section~\ref{sec:hearts}, using \cite[Theorem 4.12 and Remark 4.17]{AHR}, to the equivariant hull of Theorem~\ref{thm:kuranishi}. Its invariant ring is the completed moduli ring; the fixed diagonal family gives the quotient map from $C$.
\end{proof}

\begin{lemma}\label{lem:componentideals}
Near a point of $Z$, write
\[
M=\begin{pmatrix}
m_{11}&m_{12}&m_{13}\\
m_{21}&m_{22}&m_{23}
\end{pmatrix},
\qquad
N=\begin{pmatrix}
m_{11}&m_{12}&m_{13}&q_1\\
m_{21}&m_{22}&m_{23}&q_2
\end{pmatrix}.
\]
In the regular ring $S_c[[q_1,q_2]]$, let
\[
P_C=(q_1,q_2),\qquad
P_R=(e_1,e_2,I_2(N)).
\]
Then the defining ideal $J$ in \eqref{eq:ambientcompletedring} satisfies
\begin{equation}\label{eq:primeintersection}
J=P_C\cap P_R.
\end{equation}
Both $P_C$ and $P_R$ are prime, and
\[
P_C+P_R=(q_1,q_2,e_1,e_2,I_2(M)).
\]
\end{lemma}
\begin{proof}
The inclusion $J\subset P_C\cap P_R$ is immediate. For the reverse inclusion, first ignore the independent regular parameters $e_1,e_2$. If $\Delta_{jk}=m_{1j}m_{2k}-m_{1k}m_{2j}$, the determinantal syzygies
\[
q_1\Delta_{jk}
=m_{1j}(q_1m_{2k}-q_2m_{1k})
-m_{1k}(q_1m_{2j}-q_2m_{1j})
\]
and the analogous identities with $q_2$ show that
\[
(q_1,q_2)\cap I_2(N)
=(q_1m_{2k}-q_2m_{1k}:1\le k\le3).
\]
Indeed $I_2(N)=I_2(M)+(q_1m_{2k}-q_2m_{1k})$, and
$(q_1,q_2)\cap I_2(M)=(q_1,q_2)I_2(M)$ because the entries of $M$ lie in the coefficient ring; the displayed syzygies absorb this product. Reintroducing the independent parameters $e_1,e_2$ gives
\[
(q_1,q_2)\cap(e_1,e_2,I_2(N))
=(e_iq_j,\ q_1m_{2k}-q_2m_{1k}),
\]
which is $J$. The ideal $P_C$ is prime because its quotient is the regular ring $S_c$. The quotient by $P_R$ is the generic rank-one $2\times4$ determinantal ring over the regular base defined by $e_1=e_2=0$, hence is a normal Cohen--Macaulay domain. The formula for the sum is immediate.
\end{proof}

\begin{theorem}\label{thm:ambientgluing}
The morphism
\[
C=X\times\Gr(2,V)\longrightarrow M^{ss}
\]
is a closed immersion onto a smooth reduced irreducible component. The good moduli space $M^{ss}$ is reduced and has exactly two irreducible components,
\[
M^{ss}=R\cup C,
\qquad \dim R=10,
\qquad \dim C=13.
\]
The principal component $R$ is normal, and the scheme-theoretic intersection is
\begin{equation}\label{eq:actualschemeintersection}
R\times_{M^{ss}}C=\Sigma.
\end{equation}
In particular the normalization morphism $R^\nu\to R$ of Theorem~\ref{thm:normalization} is an isomorphism and $\Rtilde\to R$ is the small resolution itself.
\end{theorem}
\begin{proof}
At $c\in Z$, Proposition~\ref{prop:kuranishi} and Lemma~\ref{lem:componentideals} give a reduced completed local ring with two minimal primes. The same two-branch description holds at $c\in\Sigma\setminus Z$: one entry of $M$ is a unit, and eliminating it reduces the determinantal equations to the corresponding smooth rank-one chart. At $c\in C\setminus\Sigma$, either an $e_i$ is a unit or a $2\times2$ minor of $M$ is a unit. In either case \eqref{eq:ambientcompletedring} forces $q_1=q_2=0$, and there is only the corrected branch. Thus at every corrected point the quotient by $P_C=(q_1,q_2)$ is $S_c$, with the stated completed map from $C$. The morphism $C\to M^{ss}$ is proper because $C$ is projective, and it is injective on geometric points by Corollary~\ref{cor:actual}; hence it is quasi-finite and therefore finite. For a finite morphism, surjectivity after completion implies surjectivity on the local rings by faithful flatness of completion. Thus $C\to M^{ss}$ is a closed immersion and is formally a smooth component.

Along $Z$, the quotient by $P_R$ is the normal rank-one $2\times4$ determinantal branch; along $\Sigma\setminus Z$ it is smooth. The principal image $R$ contains the dense stable rational-quintic locus and has dimension ten by Theorem~\ref{thm:principalcomponent}; hence its completed local branch along the common boundary is exactly $P_R$. Away from the boundary, Theorem~\ref{thm:principalcomponent} identifies $R$ with the smooth open subset of $\Rtilde$. Thus $R$ is normal everywhere. In particular $R=R^\nu$, and the small resolution of Theorem~\ref{thm:normalization} lands directly on $R$.

The sum $P_C+P_R$ gives $S_c/(e_1,e_2,I_2(M))$, the completed local ring of $\Sigma$ in Proposition~\ref{prop:Sigma}. Together with the rank-stratum calculations above, this proves \eqref{eq:actualschemeintersection} everywhere.

Lemma~\ref{lem:pointsexhaust} shows that there are no other geometric branches. Off $C$, the point lies in the stable principal locus. The reconstruction on Artin deformation functors in Theorem~\ref{thm:principalcomponent} identifies its local moduli with the smooth open subset of $\Rtilde$. Hence the local calculations cover every point and show that $M^{ss}$ is reduced with exactly the two asserted irreducible components.
\end{proof}

\section{Higher degree}\label{sec:remaining}
The polynomial in \eqref{eq:budget} first permits correction length two at $d=7$. For $d>5$, the necessary transport and terminal-factor classification remain open here; the polynomial alone neither establishes a corrected family nor determines its moduli geometry.

\section*{Statements and Declarations}
\noindent\textbf{Competing interests.} The author has no competing interests to declare.


\begin{thebibliography}{99}
\bibitem{AlperGMS} J.~Alper,
\emph{Good moduli spaces for Artin stacks},
Ann. Inst. Fourier (Grenoble) \textbf{63} (2013), no.~6, 2349--2402;
\href{https://doi.org/10.5802/aif.2833}{doi:10.5802/aif.2833}.
\bibitem{AHR} J.~Alper, J.~Hall, and D.~Rydh,
\emph{A Luna \'etale slice theorem for algebraic stacks},
Ann. of Math. (2) \textbf{191} (2020), no.~3, 675--738;
\href{https://doi.org/10.4007/annals.2020.191.3.1}{doi:10.4007/annals.2020.191.3.1}.
\bibitem{AHLH} J.~Alper, D.~Halpern-Leistner, and J.~Heinloth,
\emph{Existence of moduli spaces for algebraic stacks},
Invent. Math. \textbf{234} (2023), no.~3, 949--1038;
\href{https://doi.org/10.1007/s00222-023-01214-4}{doi:10.1007/s00222-023-01214-4}.
\bibitem{APR} M.~Altavilla, M.~Petkovi\'c, and F.~Rota,
\emph{Moduli spaces on the Kuznetsov component of Fano threefolds of index~2},
\'Epijournal G\'eom. Alg\'ebrique \textbf{6} (2022), Art.~13.
\href{https://doi.org/10.46298/epiga.2022.7047}{doi:10.46298/epiga.2022.7047}.
\bibitem{BMT} A.~Bayer, E.~Macr\`i, and Y.~Toda,
\emph{Bridgeland stability conditions on threefolds I: Bogomolov--Gieseker type inequalities},
J. Algebraic Geom. \textbf{23} (2014), 117--163.
\bibitem{Bridgeland} T.~Bridgeland,
\emph{Stability conditions on triangulated categories},
Ann. of Math. (2) \textbf{166} (2007), no.~2, 317--345;
\href{https://doi.org/10.4007/annals.2007.166.317}{doi:10.4007/annals.2007.166.317}.
\bibitem{CHL} K.~Chung, J.~Hong, and S.~Lee,
\emph{Geometry of moduli spaces of rational curves in linear sections of Grassmannian $\Gr(2,5)$},
J. Pure Appl. Algebra \textbf{222} (2018), 868--888.
\bibitem{CKK} K.~Chung, J.~Kim, and J.-S.~Kim,
\emph{Quartic curves in the quintic del Pezzo threefold},
to appear in J. Math. Soc. Japan;
\href{https://doi.org/10.2969/jmsj/95659565}{doi:10.2969/jmsj/95659565};
\href{https://arxiv.org/abs/2505.09130v2}{arXiv:2505.09130v2} (2 June 2026).
\bibitem{ELO} A.~I.~Efimov, V.~A.~Lunts, and D.~O.~Orlov,
\emph{Deformation theory of objects in homotopy and derived categories I: general theory},
Adv. Math. \textbf{222} (2009), no.~2, 359--401;
\href{https://arxiv.org/abs/math/0702838}{arXiv:math/0702838}.
\bibitem{Keller} B.~Keller,
\emph{Introduction to $A_\infty$ algebras and modules},
Homology Homotopy Appl. \textbf{3} (2001), no.~1, 1--35;
\href{https://arxiv.org/abs/math/9910179}{arXiv:math/9910179}.
\bibitem{Li} C.~Li,
\emph{Stability conditions on Fano threefolds of Picard number~1},
J. Eur. Math. Soc. \textbf{21} (2019), no.~3, 709--726;
\href{https://doi.org/10.4171/JEMS/848}{doi:10.4171/JEMS/848};
\href{https://arxiv.org/abs/1510.04089}{arXiv:1510.04089}.
\bibitem{Meng} J.~Meng,
\emph{Quantum cohomology of a Fano quiver moduli space},
\href{https://arxiv.org/abs/2412.15987}{arXiv:2412.15987} (2024).
\bibitem{OrlovV5} D.~O.~Orlov,
\emph{Exceptional set of vector bundles on the variety $V_5$},
Moscow Univ. Math. Bull. \textbf{46} (1991), no.~5, 48--50;
Russian original: Vestnik Moskov. Univ. Ser.~1 Mat. Mekh. (1991), no.~5, 69--71.
\bibitem{PT} D.~Piyaratne and Y.~Toda,
\emph{Moduli of Bridgeland semistable objects on 3-folds and Donaldson--Thomas invariants},
J. Reine Angew. Math. \textbf{747} (2019), 175--219;
\href{https://doi.org/10.1515/crelle-2016-0006}{doi:10.1515/crelle-2016-0006}.
\bibitem{Sanna} G.~Sanna,
\emph{Rational curves and instantons on the Fano threefold $Y_5$},
Ph.D. thesis, SISSA, 2014;
\href{https://arxiv.org/abs/1411.7994}{arXiv:1411.7994}.
\bibitem{Schmidt} B.~Schmidt,
\emph{Bridgeland stability on threefolds---some wall crossings},
J. Algebraic Geom. \textbf{29} (2020), 247--283;
\href{https://arxiv.org/abs/1509.04608}{arXiv:1509.04608}.
\bibitem{Toda} Y.~Toda,
\emph{Moduli stacks of semistable sheaves and representations of Ext-quivers},
Geom. Topol. \textbf{22} (2018), 3083--3144;
\href{https://arxiv.org/abs/1710.01841}{arXiv:1710.01841}.
\bibitem{VassilievLow} D.~A.~Vassiliev,
\emph{Rank two sheaves with low discriminant on the Fano threefold of index 2 and degree 5},
\href{https://arxiv.org/abs/2510.10608v2}{arXiv:2510.10608v2} (2025; revised 4 February 2026).
\end{thebibliography}
\end{document}